\documentclass[11pt, notitlepage]{article}
\usepackage{amssymb,amsmath,comment}
\usepackage{amsthm}
\catcode`\@=11 \@addtoreset{equation}{section}
\def\thesection{\arabic{section}}

\def\theequation{\thesection.\arabic{equation}}
\catcode`\@=12
\usepackage{colortbl}
\usepackage{hyperref}
\usepackage[mathscr]{eucal}
\usepackage{epsf}
\usepackage{esint}
\usepackage{a4wide}

\newcommand{\Om} {\Omega}

\newcommand{\la} {\lambda}
\newcommand{\La} {\Lambda}
\newcommand{\noi} {\noindent}

\usepackage[all]{xy}
\catcode`\@=11
\def\theequation{\@arabic{\c@section}.\@arabic{\c@equation}}
\catcode`\@=12

\newtheorem{Theorem}{Theorem}[section]
\newtheorem{Lemma}[Theorem]{Lemma}

\newtheorem{Corollary}[Theorem]{Corollary}
\newtheorem{Remark}[Theorem]{Remark}
\newtheorem{Definition}[Theorem]{Definition}

\begin{document}

{\vspace{0.01in}}

\title{Two alternative proofs of weak Harnack inequality for mixed local and nonlocal $p$-Laplace equations with a nonhomogeneity}

\author{Prashanta Garain}

\maketitle

\begin{abstract}\noindent
We study a class of mixed local and nonlocal $p$-Laplace equations with prototype 
\[
-\Delta_p u + (-\Delta_p)^s u = f \quad \text{in } \Omega,
\]  
where $\Omega \subset \mathbb{R}^n$ is bounded and open. We provide sufficient condition on $f$ to ensure weak Harnack inequality with a tail term for sign-changing supersolutions. Two different proofs are presented, avoiding the Krylov--Safonov covering lemma and expansion of positivity: one via the John--Nirenberg lemma, the other via the Bombieri--Giusti lemma. To our knowledge, these approaches are new, even for $p = 2$ with $f \equiv 0$, and include a new proof of the reverse H\"older inequality for supersolutions. Further, we establish Harnack inequality for solutions by first deriving a local boundedness result, together with a tail estimate and an initial weak Harnack inequality.
\end{abstract}

\maketitle

\noi {Keywords: Mixed local and nonlocal p-Laplace equation; John-Nirenberg lemma; Bombieri-Giusti lemma; Local
boundedness; Reverse H\"older inequality; Harnack inequality.}

\noi{\textit{2020 Mathematics Subject Classification: 35B45, 35B65, 35D30, 35J92, 35R11.}}

\bigskip

\tableofcontents

\section{Introduction and main results}
We study a class of mixed local and nonlocal partial and integro-differential equations of the form
\begin{equation}\label{meqn}
-\mathrm{div}(\mathcal{A}(x, \nabla u)) + \mathcal{L}u = f \quad \text{in } \Omega,
\end{equation}
where $\Omega \subset \mathbb{R}^n$ is a bounded open set. We assume $0 < s < 1 < p < n$ and $f \in L_{\mathrm{loc}}^{\frac{q}{p}}(\Omega)$ with $q > n$. The map $\mathcal{A} : \Omega \times \mathbb{R}^n \to \mathbb{R}^n$ is a Carath\'eodory function satisfying the growth and coercivity conditions
\begin{equation}\label{A}
|\mathcal{A}(x,\xi)| \leq C_2 |\xi|^{p-1}, \quad \mathcal{A}(x,\xi) \xi \geq C_1 |\xi|^p,
\end{equation}
for almost every $x \in \Omega$ and all $\xi \in \mathbb{R}^n$, with positive constants $C_1, C_2$. 

The operator $\mathcal{L}$ is the nonlocal $p$-Laplace operator defined by
\begin{equation}\label{L}
\mathcal{L}(u)(x) = \mathrm{P.V.} \int_{\mathbb{R}^n} |u(x) - u(y)|^{p-2} (u(x) - u(y)) K(x,y) \, dy,
\end{equation}
where P.V. stands for the Cauchy principal value, and the kernel $K$ is measurable, symmetric in $x$ and $y$, and satisfies for some $\Lambda \geq 1$
\begin{equation}\label{K}
\frac{\Lambda^{-1}}{|x-y|^{n + ps}} \leq K(x,y) \leq \frac{\Lambda}{|x-y|^{n + ps}} \quad \text{for almost every } x,y \in \mathbb{R}^n.
\end{equation}

When $\mathcal{A}(x,\xi) = |\xi|^{p-2}\xi$, the operator $\mathrm{div}(\mathcal{A}(x, \nabla u))$ reduces to the $p$-Laplace operator $\Delta_p u$. Similarly, if $K(x,y) = |x-y|^{-n - ps}$, then $\mathcal{L}$ coincides with the fractional $p$-Laplace operator $(-\Delta_p)^s$. Consequently, equation \eqref{meqn} generalizes the mixed local and nonlocal $p$-Laplace equation
\begin{equation}\label{meqn1}
-\Delta_p u + (-\Delta_p)^s u = f.
\end{equation}

Mixed local and nonlocal PDEs have recently garnered considerable attention due to their applications in diverse physical models (see \cite{Valpi} and references therein).

When $f \equiv 0$, various regularity results have been established for the equation \eqref{meqn1}. For the linear case $p=2$, Harnack inequalities for globally nonnegative solutions have been obtained via probabilistic methods (see \cite{Chen, Fo}). In the quasilinear setting, several regularity results including local boundedness and Harnack inequalities have been proven using purely analytic techniques in \cite{GKK, GKtams}.

For nontrivial right-hand sides $f$, recent advances include results for $p=2$ where $f=gu$ with $g$ belonging to the Kato class, establishing Harnack inequalities for globally nonnegative solutions of \eqref{meqn1} via probabilistic arguments \cite{Siva}. Additionally boundedness and H\"older regularity results appear in \cite{Valcpde, Valmz, Valjde}. Weak Harnack inequalities for sign-changing supersolutions with bounded, nonpositive $g$ are shown in \cite{Biagicvpde}.

When $p \neq 2$, assuming various conditions on $f$, boundedness, higher H\"older and $C^{1,\alpha}$ regularity have been obtained in \cite{Hong, FMma, GLcvpde, Zhangbdd}. Lipschitz potential estimates for $f \in L^1_{\mathrm{loc}}(\Omega)$ are found in \cite{Biswasnodea}, while global gradient estimates under divergence-form $f$ appear in \cite{Byuncvpde2}. For measure data, potential and gradient estimates are available in \cite{Byuncvpde, Iwonajlms}. Moreover, weak Harnack inequalities for nonlinear source terms of the form $f = g(x)|u|^{p-2} u$ with bounded, nonpositive $g$ are presented in \cite{Garainna}.

In this work, we provide sufficient conditions on the nonhomogeneity $f$ that guarantees weak Harnack inequality, local boundedness and Harnack inequality for sign-changing solutions of \eqref{meqn}. Our main results in this article reads as follows:

First, we have the following weak Harnack inequality, which is known for the equation \eqref{meqn1} with $f\equiv 0$ in \cite[Theorem 8.4]{GKtams}.

\begin{Theorem}[Weak Harnack Inequality]\label{mthm}
Let $x_0 \in \mathbb{R}^n$ and $0 < r \leq 1$ satisfy $r < \frac{R}{2}$. Assume $0 < s < 1 < p < n$ and $f \in L^{\frac{q}{p}}(B_R(x_0))$ with $q > n$. Suppose $u$ is a weak supersolution of \eqref{meqn} with $u \geq 0$ in $B_R(x_0) \subset \Omega$. Then, for every $0 < \eta < \kappa(p-1)$, there exists a constant
$
c = c(n,p,q,s,C_1,C_2,\Lambda,\eta) > 0
$
such that
\begin{equation}\label{mthmest}
\left( \fint_{B_{\frac{r}{2}}(x_0)} u^\eta \right)^{\frac{1}{\eta}} \leq c \left( \inf_{B_{\frac{r}{2}}(x_0)} u + r^{p'} R^{-p'} \mathrm{Tail}(u_-; x_0, R) + r^{\frac{p-n}{p-1}} R^{\frac{n(q-p)}{q(p-1)}} \|f\|_{L^{\frac{q}{p}}(B_R(x_0))}^{\frac{1}{p-1}} \right),
\end{equation}
where the tail term $\mathrm{Tail}(u_-; x_0, R)$ is defined in \eqref{tail}.
\end{Theorem}

For the homogeneous case $f \equiv 0$, the classical proof of Theorem \ref{mthm} relies heavily on the Krylov–Safonov covering lemma and the expansion of positivity technique. Here, we provide two different proofs that circumvent these classical tools and the De Giorgi framework. Moreover, as far as we are aware, these proofs are new even for $p=2$ with $f\equiv 0$.

The first approach adapts Moser's iteration scheme \cite{Moser} to the nonlinear nonlocal setting, using the John–Nirenberg inequality (Theorem~\ref{JN-lem}). This method, previously employed in the purely local \cite{Moser} and purely nonlocal contexts \cite{Chakeretal}, centers on a logarithmic estimate (Lemma~\ref{loglem}) connecting integral averages with positive and negative powers of supersolutions (see \eqref{neg-poseqn}).

A key ingredient is controlling averages of negative powers through the infimum of the supersolution (Lemma~\ref{Cor}), leading to a preliminary weak Harnack inequality (estimate~\eqref{pwh1}). This is then refined using a reverse H\"older inequality (Lemma~\ref{RevH}) to complete the proof of Theorem~\ref{mthm}.

The second approach leverages a modified Bombieri–Giusti lemma (Lemma~\ref{Bom-Giu}) originally developed in \cite{BoGi}, which bypasses both the John–Nirenberg and Krylov–Safonov lemmas. This powerful tool has been applied in various local and nonlocal problems \cite{Kasscpde, GKjde, Ahmedrmi, Nakamuracvpde}. The proof integrates three key components: a supremum estimate for inverse of supersolutions (Lemma~\ref{Cor}), a reverse H\"older inequality (Lemma~\ref{RevH}), and a tailored logarithmic estimate (Lemma~\ref{logest2}).

Both approaches rest on precise energy estimates (see Lemma~\ref{inveng}, equations \eqref{negeng1} and \eqref{negeng2}), obtained via two crucial algebraic inequalities \eqref{alg2} and \eqref{alg1} respectively, which handles the nonlinear structure and coupling between local and nonlocal terms. Such energy estimates combined with the Sobolev embeddings and Moser's iteration technique proved Lemma \ref{Cor} and Lemma \ref{RevH}. Additionally, a weighted Poincaré inequality (Lemma~\ref{wgtPoin}) is instrumental in deriving the final logarithmic estimates in Lemma \ref{logest2}. Notably, the proof of the reverse H\"older inequality in Lemma~\ref{RevH} differs from that in \cite[estimate (8.16)]{GKtams}.

Our second main result establishes local boundedness, and its proof follows the approach of \cite[Theorem 4.2]{GKtams}, originally carried out for the case $f \equiv 0$ in equation \eqref{meqn1}. Further, local boundedness result for nonzero $f$ can be found in \cite[Theorem 1.2]{GLcvpde} when $p\geq 2$.
\begin{Theorem}[Local boundedness]\label{mthmbdd}
Let $x_0 \in \mathbb{R}^n$ and $0 < r \leq 1$ {be such that $B_R(x_0)\subset\Omega$ with $r\leq R$}. Assume that $0 < s < 1 < p < n$, $\kappa = \frac{n}{n-p}$, and $f \in L^{\frac{q}{p}}(B_R(x_0))$ with $q > n$. Suppose that $u$ is a weak subsolution of \eqref{meqn}. Then, there exists a constant 
$
c = c(n,p,q,s,C_1,C_2,\Lambda) > 0
$
such that
\begin{equation}\label{mthmbddest}
\sup_{B_{\frac{r}{2}}(x_0)} u
\leq 
\delta \, \mathrm{Tail}\Big(u_+; x_0, \frac{r}{2}\Big)
+ c \, \delta^{-\frac{(p-1)\kappa}{p(\kappa-1)}}
\left( \fint_{B_r(x_0)} u_+^p \, dx \right)^{\frac{1}{p}}
+ r^{\frac{p-n}{p-1}} R^{\frac{n(q-p)}{q(p-1)}}
\|f\|_{L^{\frac{q}{p}}(B_R(x_0))}^{\frac{1}{p-1}},
\end{equation}
where the tail term $\mathrm{Tail}\big(u_+; x_0, \frac{r}{2}\big)$ is defined in \eqref{tail}.
\end{Theorem}
Finally, we have the following Harnack inequality, which is known for the equation \eqref{meqn1} with $f\equiv 0$ in \cite[Theorem 8.3]{GKtams}. The proof is done by combining the above local boundedness result along with a tail estimate and a preliminary weak Harnack inequality.
\begin{Theorem}[Harnack Inequality]\label{mthm2}
Let $x_0 \in \mathbb{R}^n$ and $0 < r \leq 1$ satisfy $r < \frac{R}{2}$. Assume that $0 < s < 1 < p < n$ and $f \in L^{\frac{q}{p}}(B_R(x_0))$ with $q > n$. Suppose that $u$ is a weak solution of \eqref{meqn} with $u \geq 0$ in $B_R(x_0) \subset \Omega$. Then, there exists a constant 
$
c = c(n,p,q,s,C_1,C_2,\Lambda) > 0
$
such that
\begin{equation}\label{mthm2est}
\sup_{B_{\frac{r}{2}}(x_0)} u 
\leq 
c \left( 
\inf_{B_r(x_0)} u 
+ r^{p'} R^{-p'} \, \mathrm{Tail}(u_-; x_0, R) 
+ r^{\frac{p-n}{p-1}} R^{\frac{n(q-p)}{q(p-1)}}\|f\|_{L^{\frac{q}{p}}(B_R(x_0))}^{\frac{1}{p-1}} 
\right),
\end{equation}
where the tail term $\mathrm{Tail}(u_-; x_0, R)$ is defined in \eqref{tail}.
\end{Theorem}

\subsection*{Notation and Assumptions}
Throughout the rest of the paper, we shall use the following conventions unless stated otherwise:
\begin{itemize}
    \item We fix $0 < s < 1 < p < n$ and define $\kappa = \frac{n}{n-p}$.
    \item The domain $\Omega$ denotes a bounded open subset of $\mathbb{R}^n$.
    \item The source function satisfies $f \in L_{\mathrm{loc}}^{\frac{q}{p}}(\Omega)$ with $q > n$.
    \item For $k \in \mathbb{R}$, we denote $k_\pm = \max\{0, \pm k\}$.
    \item The conjugate exponent of $l>1$ is denoted by $l'$ and defined by $l'=\frac{l}{l-1}$.
    \item For a point $x_0 \in \mathbb{R}^n$ and radius $r > 0$, we denote the ball
    \[
    B_r(x_0) = \{ y \in \mathbb{R}^n : |y - x_0| < r \}.
    \]
    \item The underlying space of a weak subsolution or supersolution or solution of \eqref{meqn} in $\Omega$ will be $W_{\mathrm{loc}}^{1,p}(\Omega) \cap L^{p-1}_{ps}(\mathbb{R}^n)$.
    \item We denote by $d\mu = K(x,y) \, dx dy$ is defined using the kernel $K$ from \eqref{K}.
    \item The Lebesgue measure of a measurable set $S$ is denoted by $|S|$.
    \item The barred integral symbol denotes the corresponding integral average.
    \item For a measurable function $u$ and weight $\phi$, we write
    \[
    (u)_{\phi, B_r(x_0)} = \frac{\int_{B_r(x_0)} u \phi \, dx}{\int_{B_r(x_0)} \phi \, dx}.
    \]
    \item The function $J_p : \mathbb{R} \to \mathbb{R}$ is defined by
    \[
    J_p(t) = |t|^{p-2} t.
    \]
    \item The constants $C_1,C_2$ and $\La$ will be the same constants as found in \eqref{A} and \eqref{K} respectively.
    \item If a constant $C$ depends on parameters $r_1, r_2, \dots, r_k$, we write
    \[
    C = C(r_1, r_2, \dots, r_k).
    \]
\end{itemize}

\subsection*{Organization of the paper}

The paper is organized as follows. Section 2 introduces the functional framework and recalls auxiliary results. In Section 3, we establish the required energy estimates. Section 4 is devoted to proving supersolution and tail estimates. Finally, Section 5 contains the proofs of the main results.
\section{Functional setting and auxiliary results}
\subsection{Functional setting}
Let $\Omega \subset \mathbb{R}^n$ be a bounded open set and let $0 < s < 1 < p < \infty$. We consider the space $W^{1,p}(\Om)$ as the standard Sobolev space \cite{Evans}. In order to study mixed local and nonlocal problems, it is convenient to consider the space $W_0^{1,p}(\Omega)$ (see \cite{Hong} and the references therein) defined as the closure of $C_c^\infty(\Om)$ with respect to the global norm $\Big(\int_{\mathbb{R}^n}|\nabla u|^p\,dx\Big)^\frac{1}{p}$.  
In addition, we make use of the fractional Sobolev space $W^{s,p}(\Omega)$, defined by (see \cite{Hitchhiker'sguide})
\[
W^{s,p}(\Omega) = \left\{ u \in L^p(\Omega) : \int_{\Omega} \int_{\Omega} \frac{|u(x) - u(y)|^p}{|x - y|^{n + ps}} \, dx \, dy < \infty \right\},
\]
and equipped with the norm
\[
\|u\|_{W^{s,p}(\Omega)} = \left( \int_{\Omega} |u(x)|^p \, dx + \int_{\Omega} \int_{\Omega} \frac{|u(x) - u(y)|^p}{|x - y|^{n + ps}} \, dx \, dy \right)^{\frac{1}{p}}.
\]

The fractional Sobolev space with zero boundary values is given by
\[
W_0^{s,p}(\Omega) = \left\{ u \in W^{s,p}(\mathbb{R}^n) : u = 0 \text{ in } \mathbb{R}^n \setminus \Omega \right\}.
\]

For $0 < s \leq 1$, the local Sobolev space $W^{s,p}_{\mathrm{loc}}(\Omega)$ consists of all functions $u$ such that $u \in W^{s,p}(\Omega')$ for every open subset $\Omega' \Subset \Omega$, where $\Omega' \Subset \Omega$ means that the closure $\overline{\Omega'}$ is compactly contained in $\Omega$.  

The next lemma asserts that, under appropriate smoothness conditions on $\Omega$, the classical Sobolev space embeds continuously into the fractional Sobolev space (cf. \cite[Proposition 2.2]{Hitchhiker'sguide}).  

\begin{Lemma}\label{locnon}
Let $\Omega$ be a smooth bounded domain in $\mathbb{R}^n$, with $1 < p < \infty$ and $0 < s < 1$. There exists a positive constant $C = C(n,p,s)$ such that
\[
\|u\|_{W^{s,p}(\Omega)} \leq C \|u\|_{W^{1,p}(\Omega)}
\]
for every $u \in W^{1,p}(\Omega)$.
\end{Lemma}

A related result for fractional Sobolev spaces with zero boundary conditions is given next (see \cite[Lemma 2.1]{Silva}). Unlike Lemma \ref{locnon}, this result holds for any bounded domain since functions in $W_0^{1,p}(\Omega)$ admit a zero extension outside $\Omega$.

\begin{Lemma}\label{locnon1}
Let $\Omega$ be a bounded domain in $\mathbb{R}^n$, with $1 < p < \infty$ and $0 < s < 1$. There exists a positive constant $C = C(n,p,s,\Omega)$ such that
\[
\int_{\mathbb{R}^n} \int_{\mathbb{R}^n} \frac{|u(x) - u(y)|^p}{|x - y|^{n + ps}} \, dx \, dy \leq C \int_{\Omega} |\nabla u|^p \, dx
\]
for every $u \in W_0^{1,p}(\Omega)$.
\end{Lemma}

For $0 < s < 1 < p < \infty$, we define the tail space
\[
L^{p-1}_{ps}(\mathbb{R}^n) = \left\{ v \in L^{p-1}_{\mathrm{loc}}(\mathbb{R}^n) : \int_{\mathbb{R}^n} \frac{|v(y)|^{p-1}}{(1 + |y|)^{n + ps}} \, dy < \infty \right\}.
\]
Furthermore, for $x_0 \in \mathbb{R}^n$ and $\rho > 0$, the mixed local and nonlocal tail quantity, which appears in our main results, is defined by
\begin{equation}\label{tail}
\mathrm{Tail}(v; x_0, \rho) = \left( \rho^p \int_{\mathbb{R}^n \setminus B_\rho(x_0)} \frac{|v(y)|^{p-1}}{|y - x_0|^{n + ps}} \, dy \right)^{\frac{1}{p-1}}.
\end{equation}
By construction, for any $v \in L^{p-1}_{ps}(\mathbb{R}^n)$, the tail quantity $\mathrm{Tail}(v; x_0, \rho)$ is finite for every $x_0 \in \mathbb{R}^n$ and $\rho > 0$.

Next, we define the notion of weak subsolutions or supersolutions for equation \eqref{meqn}.

\begin{Definition}\label{subsupsolution}
A function $u \in W_{\mathrm{loc}}^{1,p}(\Omega) \cap L^{p-1}_{ps}(\mathbb{R}^n)$ is a weak subsolution (or supersolution) of \eqref{meqn} if for every $\Omega' \Subset \Omega$ and every nonnegative test function $\phi \in W_0^{1,p}(\Om')$, we have
\begin{equation}\label{weaksubsupsoln}
\int_{\Omega'} \mathcal{A}(x, \nabla u) \cdot \nabla \phi \, dx + \int_{\mathbb{R}^n} \int_{\mathbb{R}^n} J_p(u(x) - u(y)) (\phi(x) - \phi(y)) \, d\mu \leq (\text{ or }\geq) \int_{\Omega'} f \phi \, dx,
\end{equation}
where $\mathcal{A}$ is as defined in \eqref{A} with
\[
J_p(u(x) - u(y)) = |u(x) - u(y)|^{p-2} (u(x) - u(y))
\quad \text{and} \quad
d\mu = K(x,y) \, dx \, dy,
\]
where $K$ is as defined in \eqref{K}. Moreover, we say that $u \in W_{\mathrm{loc}}^{1,p}(\Omega) \cap L^{p-1}_{ps}(\mathbb{R}^n)$ is a weak solution of \eqref{meqn} if the equality in \eqref{weaksubsupsoln} holds for every $\phi\in W_0^{1,p}(\Om')$ without any sign restriction.  
\end{Definition}

It immediately follows from Definition \ref{subsupsolution} that $u$ is a weak subsolution or supersolution of \eqref{meqn} if and only if, for any constant $c \in \mathbb{R}$, the translated function $u + c$ is also a weak subsolution or supersolution of \eqref{meqn}. 
\subsection{Auxiliary results}

The following iteration lemma from \cite[Lemma 1.1]{GGactamath} will be useful for us.  

\begin{Lemma}\label{ite}
Let $0 \leq T_0 \leq t \leq T_1$ and assume that $f : [T_1, T_2] \to [0, \infty)$ is a nonnegative bounded function. Suppose that for $T_0 \leq t < s \leq T_1$, we have  
$$
f(t) \leq A (t - s)^{-\alpha} + B + \theta f(s),
$$
where $A, B, \alpha, \theta$ are nonnegative constants and $\theta < 1$. Then, there exists a constant $c = c(\alpha, \theta)$ such that for every $\rho, R$ with $T_0 \leq \rho < R \leq T_1$, we have  
$$
f(\rho) \leq c \big( A (R - \rho)^{-\alpha} + B \big).
$$
\end{Lemma}

The following version of the Gagliardo-Nirenberg-Sobolev inequality will be useful; see \cite[Corollary 1.57]{Maly}.

\begin{Lemma}\label{emb}
Let $\Omega \subset \mathbb{R}^n$ be an open set with finite measure $|\Omega| < \infty$, and set $\kappa = \frac{n}{n - p}$ for $1 < p < n$. There exists a positive constant $C = C(n,p)$ such that
\begin{equation}\label{e.friedrich}
\left( \int_\Omega |u|^{p \kappa} \, dx \right)^{\frac{1}{p \kappa}} \leq C \left( \int_\Omega |\nabla u|^p \, dx \right)^{\frac{1}{p}}
\end{equation}
holds for every $u \in W_0^{1,p}(\Omega)$.
\end{Lemma}

A weighted Poincaré inequality, which follows from a careful inspection of the proof in \cite[Corollary 3]{Kmnine} and a change of variables, is stated below.

\begin{Lemma}\label{wgtPoin}
Let $r > 0$ and let $\phi : B_r(x_0) \to [0,\infty)$ be a radially decreasing function with $\phi \equiv 1$ in $B_{\frac{r}{2}}(x_0)$. Then there exists a positive constant $C = C(n,p)$ such that
\begin{equation}\label{wgtine}
\int_{B_r(x_0)} |u - (u)_{\phi, B_r(x_0)}|^p \, \phi \, dx \leq C r^p \int_{B_r(x_0)} |\nabla u|^p \, \phi \, dx
\end{equation}
holds for every $u \in W^{1,p}(B_r(x_0))$.
\end{Lemma}

The following theorem is a classical John-Nirenberg Lemma (see \cite{JN}).

\begin{Theorem}\label{JN-lem}
Let $w \in L_{\mathrm{loc}}^1(\Omega)$ and suppose there exists a constant $K$ such that
\[
\fint_{B_r(x_0)} |w(x) - (w)_{1, B_r(x_0)}| \, dx \leq K
\]
whenever $B_{2r}(x_0) \subset \Omega$. Then there exists a constant $\nu = \nu(n) > 0$ such that
\[
\fint_{B_r(x_0)} e^{\nu |w(x) - (w)_{1, B_r(x_0)}|} \, dx \leq 2,
\]
whenever $B_{2r}(x_0) \subset \Omega$.
\end{Theorem}

We also employ a variant of an abstract lemma originally due to Bombieri and Giusti \cite{BoGi}, adapted in \cite[Lemma 2.11]{KK}.

\begin{Lemma}\label{Bom-Giu}
Let $\nu$ be a Borel measure and let $\theta$, $A$, and $\gamma$ be positive constants, with $0 < \delta < 1$ and $0 < q \leq \infty$. Suppose $U_\sigma$ are bounded measurable sets with $U_{\sigma'} \subseteq U_\sigma$ for $0 < \delta \leq \sigma' < \sigma \leq 1$. If $q < \infty$, assume the doubling condition
\[
\nu(U_1) \leq A \nu(U_\delta).
\]
Let $f$ be a positive measurable function on $U_1$ satisfying the reverse H\"older inequality
\[
\left( \fint_{U_{\sigma'}} f^q \, d\nu \right)^{\frac{1}{q}} \leq \left( \frac{A}{(\sigma - \sigma')^\theta} \fint_{U_\sigma} f^s \, d\nu \right)^{\frac{1}{s}}
\]
for $0 < s < q$, and
\[
\nu\left( \{ x \in U_1 : \log f > \lambda \} \right) \leq \frac{A \nu(U_\delta)}{\lambda^\gamma}
\]
for all $\lambda > 0$. Then
\[
\left( \int_{U_\delta} f^q \, d\nu \right)^{\frac{1}{q}} \leq C,
\]
where $C$ depends only on $\theta$, $\delta$, $\gamma$, $q$, and $A$.
\end{Lemma}

Finally, we state some algebraic inequalities that are crucial to obtain energy estimates for weak supersolutions of \eqref{meqn}. The inequalities \eqref{alg2} and \eqref{alg1} below are taken from \cite[Lemma 3.4]{Chakeretal} and \cite[Lemma 2.9]{BGK} respectively.

\begin{Lemma}\label{alg}
Let $a,b > 0$ and $p \in (1, \infty)$.

\begin{enumerate} 
    \item[(a)] For any $\tau_1, \tau_2 \in [0,1]$ and any $\epsilon > p-1$, there exist constants $c_i = c_i(p, \epsilon) > 0$, $i=1,2$, bounded when $\epsilon$ is bounded away from $p-1$, such that
    \begin{equation}\label{alg2}
    J_p(b - a) \big(\tau_1^p a^{-\epsilon} - \tau_2^p b^{-\epsilon} \big) \geq c_1 \left| \tau_2 b^{\frac{p-1-\epsilon}{p}} - \tau_1 a^{\frac{p-1-\epsilon}{p}} \right|^p - c_2 |\tau_2 - \tau_1|^p \left( b^{p-1-\epsilon} + a^{p-1-\epsilon} \right).
    \end{equation}

    \item[(b)] For any $\tau_1, \tau_2 \geq 0$ and any $\epsilon \in (0, p-1)$, there exists a constant $C = C(p) > 1$ such that
    \begin{equation}\label{alg1}
    J_p(b - a) \big(\tau_1^p a^{-\epsilon} - \tau_2^p b^{-\epsilon} \big) \geq \frac{\zeta(\epsilon)}{C} \left| \tau_2 b^{\frac{p-1-\epsilon}{p}} - \tau_1 a^{\frac{p-1-\epsilon}{p}} \right|^p - \left( \zeta(\epsilon) + 1 + \epsilon^{1-p} \right) |\tau_2 - \tau_1|^p \left( b^{p-1-\epsilon} + a^{p-1-\epsilon} \right),
    \end{equation}
    where $\zeta(\epsilon) = \epsilon \left(\frac{p}{p-1-\epsilon}\right)^p$. Moreover, if $0 < p-1-\epsilon < 1$, one may choose $\zeta(\epsilon) = \frac{\epsilon p^p}{p-1-\epsilon}$.
\end{enumerate}
\end{Lemma}

\section{Energy estimates}
\subsection{Energy estimate for supersolutions}
In this section, we establish energy estimates for weak supersolutions of \eqref{meqn}. The first one provides an energy estimate for negative exponents and the second one for positive exponents, respectively.

\begin{Lemma}\label{inveng}
Suppose that $0 < r \leq 1$ with $r < R$. Let $u$ be a weak supersolution of \eqref{meqn} such that 
$u \geq 0$ in $B_R(x_0) \subset \Omega$.
Assume $v = u + d$, where $d > 0$, and let $\psi \in C_c^\infty(B_r(x_0))$ be nonnegative in $B_r(x_0)$.
\begin{enumerate}
\item[a.] Then for any $\epsilon > p - 1$ and every $\delta > 0$, there exists a positive constant $C = C(p,C_1, C_2, \Lambda,  \epsilon)$ such that
\begin{equation}\label{negeng1}
\begin{split}
&\int_{B_r(x_0)} \psi^p \Big| \nabla \Big( v^{-\frac{\alpha}{p}} \Big) \Big|^p \, dx \\
&\leq C \frac{\alpha^p}{\epsilon} \Bigg\{ \epsilon^{1-p} \int_{B_r(x_0)} v^{-\alpha} |\nabla \psi|^p \, dx + \int_{B_r(x_0)} \int_{B_r(x_0)} \big( v^{-\alpha}(x) + v^{-\alpha}(y) \big) |\psi(x) - \psi(y)|^p \, d\mu \\
&\quad + \Bigg( \sup_{x \in \mathrm{supp} \, \psi} \int_{\mathbb{R}^n \setminus B_r(x_0)} \frac{dy}{|x - y|^{n + ps}} +d^{1-p} \sup_{x \in \mathrm{supp} \, \psi} \int_{\mathbb{R}^n \setminus B_R(x_0)} \frac{u_-(y)^{p-1}}{|x-y|^{n + ps}} \, dy \Bigg) \int_{B_r(x_0)} v^{-\alpha} \psi^p \, dx \\
&\quad + d^{1-p} \|f\|_{L^{\frac{q}{p}}(B_r(x_0))} \Bigg( \delta \big\| \psi v^{-\frac{\alpha}{p}} \big\|_{L^{p \kappa}(B_r(x_0))}^p + \delta^{-\frac{n}{q-n}} \big\| \psi v^{-\frac{\alpha}{p}} \big\|_{L^p(B_r(x_0))}^p \Bigg) \Bigg\},
\end{split}
\end{equation}
where $\alpha = \epsilon - p + 1$. Moreover, the constant $C$ in \eqref{negeng1} is bounded as long as $\epsilon$ is bounded away from $p-1$.

\item[b.] Then for any $\epsilon \in (0, p - 1)$ and every $\delta > 0$, there exists a positive constant $C = C(C_1, C_2, \Lambda, p)$ such that
\begin{equation}\label{negeng2}
\begin{split}
&\int_{B_r(x_0)} \psi^p \Big| \nabla \Big( v^{\frac{\beta}{p}} \Big) \Big|^p \, dx \\
&\leq C \frac{\beta^p}{\epsilon} \Bigg\{ \epsilon^{1-p} \int_{B_r(x_0)} v^{\beta} |\nabla \psi|^p \, dx \\
&\quad + \Big( \zeta(\epsilon) + 1 + \epsilon^{1-p} \Big) \int_{B_r(x_0)} \int_{B_r(x_0)} \big( v^{\beta}(x) + v^{\beta}(y) \big) |\psi(x) - \psi(y)|^p \, d\mu \\
&\quad + \Bigg( \sup_{x \in \mathrm{supp} \, \psi} \int_{\mathbb{R}^n \setminus B_r(x_0)} \frac{dy}{|x - y|^{n + ps}} + d^{1-p} \sup_{x \in \mathrm{supp} \, \psi} \int_{\mathbb{R}^n \setminus B_R(x_0)} \frac{u_-(y)^{p-1}}{|x-y|^{n + ps}} \, dy \Bigg) \int_{B_r(x_0)} v^{\beta} \psi^p \, dx \\
&\quad + d^{1-p} \|f\|_{L^{\frac{q}{p}}(B_r(x_0))} \Bigg( \delta \big\| \psi v^{\frac{\beta}{p}} \big\|_{L^{p \kappa}(B_r(x_0))}^p + \delta^{-\frac{n}{q-n}} \big\| \psi v^{\frac{\beta}{p}} \big\|_{L^p(B_r(x_0))}^p \Bigg) \Bigg\},
\end{split}
\end{equation}
where $\beta = p - 1 - \epsilon$. Here, $\zeta(\epsilon) = \frac{\epsilon p^p}{\beta^p}$. Moreover, if $0 < \beta < 1$, then we may choose $\zeta(\epsilon) = \frac{\epsilon p^p}{\beta}$.
\end{enumerate}
\end{Lemma}

\begin{proof}
We prove both inequalities simultaneously. To this end, let $\epsilon > 0$ be such that $\epsilon \neq p - 1$, and define $\gamma = \epsilon - p + 1$. By choosing $\phi = v^{-\epsilon} \psi^{p}$ as a test function in \eqref{weaksubsupsoln}, we obtain
\begin{equation}\label{IJK}
I + J \leq K,
\end{equation}
where
\[
I = \int_{\Omega} \mathcal{A}(x, \nabla v) \nabla \left(-v^{-\epsilon} \psi^{p}\right) \, dx,
\]
\[
J = \int_{\mathbb{R}^n} \int_{\mathbb{R}^n} J_p(v(x) - v(y)) \bigl(v^{-\epsilon}(y) \psi^{p}(y) - v^{-\epsilon}(x) \psi^{p}(x)\bigr) \, d\mu,
\quad \text{and} \quad
K = - \int_{\Omega} f v^{-\epsilon} \psi^{p} \, dx.
\]

\textbf{Estimate of $I$:} To estimate $I$, we use the hypothesis \eqref{A} on $\mathcal{A}$ along with Young's inequality, yielding
\[
p \psi^{p-1} \nabla \psi \cdot \mathcal{A}(x, \nabla v) v^{-\epsilon} \leq \frac{C_1 \epsilon}{2} \psi^{p} v^{-\epsilon - 1} |\nabla v|^{p} + c(p, C_1, C_2) \epsilon^{1-p} v^{-\gamma} |\nabla \psi|^{p}.
\]
Further applying the hypothesis \eqref{A} on $\mathcal{A}$ and the above estimate, we get
\begin{equation}\label{estI}
I \geq c_1(p, C_1, C_2) \frac{\epsilon}{|\gamma|^{p}} \int_{B_r(x_0)} \psi^{p} \left| \nabla \left( v^{-\frac{\gamma}{p}} \right) \right|^{p} \, dx - c(p, C_1, C_2) \epsilon^{1-p} \int_{B_r(x_0)} v^{-\gamma} |\nabla \psi|^{p} \, dx.
\end{equation}

\textbf{Estimate of $J$:} We first write
\[
J = J_1 - 2 J_2,
\]
where
\[
J_1 = \int_{B_r(x_0)} \int_{B_r(x_0)} J_p(v(x) - v(y)) \bigl(v^{-\epsilon}(y) \psi^{p}(y) - v^{-\epsilon}(x) \psi^{p}(x)\bigr) \, d\mu,
\]
and
\[
J_2 = \int_{x \in B_r(x_0)} \int_{y \in \mathbb{R}^n \setminus B_r(x_0)} J_p(v(x) - v(y)) v^{-\epsilon}(x) \psi^{p}(x) \, d\mu.
\]

\textbf{Estimate of $J_1$:} For $\epsilon > p - 1$, applying the inequality \eqref{alg2} from Lemma \ref{alg} yields
\begin{equation}\label{estJ1}
\begin{split}
J_1 &\geq c_1 \int_{B_r(x_0)} \int_{B_r(x_0)} \left| \psi(x) v^{-\frac{\gamma}{p}}(x) - \psi(y) v^{-\frac{\gamma}{p}}(y) \right|^{p} \, d\mu \\
&\quad - c_2 \int_{B_r(x_0)} \int_{B_r(x_0)} \bigl(v^{-\gamma}(x) + v^{-\gamma}(y)\bigr) |\psi(x) - \psi(y)|^{p} \, d\mu,
\end{split}
\end{equation}
where the constants $c_i = c_i(p, \epsilon) > 0$ for $i=1,2$ remain bounded as long as $\epsilon$ stays away from $p - 1$.

For $\epsilon \in (0, p - 1)$, by using the inequality \eqref{alg1} in Lemma \ref{alg}, we obtain
\begin{equation}\label{estJ11}
\begin{split}
J_1 &\geq \frac{\zeta(\epsilon)}{C(p)} \int_{B_r(x_0)} \int_{B_r(x_0)} \left| \psi(x) v^{-\frac{\gamma}{p}}(x) - \psi(y) v^{-\frac{\gamma}{p}}(y) \right|^{p} \, d\mu \\
&\quad - \left(\zeta(\epsilon) + 1 + \epsilon^{1-p}\right) \int_{B_r(x_0)} \int_{B_r(x_0)} \bigl(v^{-\gamma}(x) + v^{-\gamma}(y)\bigr) |\psi(x) - \psi(y)|^{p} \, d\mu,
\end{split}
\end{equation}
where $\zeta(\epsilon) = \frac{\epsilon p^{p}}{(p - 1 - \epsilon)^{p}}$. Moreover, if $0 < p - 1 - \epsilon < 1$, then $\zeta(\epsilon)$ can be chosen as $\frac{\epsilon p^{p}}{p - 1 - \epsilon}$.

\textbf{Estimate of $J_2$:} For any $x \in B_R(x_0)$ and $y \in \mathbb{R}^n$, we note that
\[
J_p(v(x) - v(y)) \leq c(p) \bigl(v^{p-1}(x) + u_-^{p-1}(y)\bigr), \quad \text{and} \quad v^{-\epsilon}(x) \leq d^{1-p} v^{-\gamma}(x).
\]
Since $u \geq 0$ in $B_R(x_0)$, the above imply
\begin{equation}\label{estJ2}
\begin{split}
J_2 &\leq C(\Lambda, p) \Bigg( \sup_{x \in \mathrm{supp} \, \psi} \int_{\mathbb{R}^n \setminus B_r(x_0)} \frac{dy}{|x - y|^{n + ps}} + d^{1-p} \int_{\mathbb{R}^n \setminus B_R(x_0)} \frac{u_-^{p-1}(y)}{|x-y|^{n + ps}} \, dy \Bigg) \\
&\quad \times \int_{B_r(x_0)} v^{-\gamma} \psi^{p} \, dx.
\end{split}
\end{equation}

Therefore, if $\epsilon > p - 1$, using \eqref{estJ1} and \eqref{estJ2} together, we conclude
\begin{equation}\label{estJ}
\begin{split}
J &\geq c_1 \int_{B_r(x_0)} \int_{B_r(x_0)} \left| \psi(x) v^{-\frac{\gamma}{p}}(x) - \psi(y) v^{-\frac{\gamma}{p}}(y) \right|^{p} \, d\mu \\
&\quad - c_2 \int_{B_r(x_0)} \int_{B_r(x_0)} \bigl(v^{-\gamma}(x) + v^{-\gamma}(y)\bigr) |\psi(x) - \psi(y)|^{p} \, d\mu \\
&\quad - 2 C(\Lambda, p) \Bigg( \sup_{x \in \mathrm{supp} \, \psi} \int_{\mathbb{R}^n \setminus B_r(x_0)} \frac{dy}{|x - y|^{n + ps}} + d^{1-p} \int_{\mathbb{R}^n \setminus B_R(x_0)} \frac{u_-^{p-1}(y)}{|x - y|^{n + ps}} \, dy \Bigg) \\
&\quad \times \int_{B_r(x_0)} v^{-\gamma} \psi^{p} \, dx,
\end{split}
\end{equation}
where the constants $c_i = c_i(p, \epsilon) > 0$ for $i=1,2$ remain bounded as long as $\epsilon$ stays away from $p - 1$.
Further, if $\epsilon \in (0, p-1)$, then using \eqref{estJ11} and \eqref{estJ2}, we obtain
\begin{equation}\label{estJnew}
\begin{split}
J &\geq \frac{\zeta(\epsilon)}{C(p)} \int_{B_r(x_0)} \int_{B_r(x_0)} \left| \psi(x) v^{-\frac{\gamma}{p}}(x) - \psi(y) v^{-\frac{\gamma}{p}}(y) \right|^p \, d\mu \\
&\quad - \Big(\zeta(\epsilon) + 1 + \epsilon^{1-p} \Big) \int_{B_r(x_0)} \int_{B_r(x_0)} \bigl( v^{-\gamma}(x) + v^{-\gamma}(y) \bigr) |\psi(x) - \psi(y)|^p \, d\mu \\
&\quad - 2 C(\Lambda, p) \Bigg( \sup_{x \in \mathrm{supp} \, \psi} \int_{\mathbb{R}^n \setminus B_r(x_0)} \frac{dy}{|x - y|^{n + ps}} + d^{1-p} \int_{\mathbb{R}^n \setminus B_R(x_0)} \frac{u_-(y)^{p-1}}{|x - y|^{n + ps}} \, dy \Bigg) \\
&\quad \times \int_{B_r(x_0)} v^{-\gamma} \psi^{p} \, dx.
\end{split}
\end{equation}

\textbf{Estimate of $K$:} Using H\"older's inequality, for every $\delta > 0$, we have
\begin{equation}\label{estK}
\begin{split}
|K| &\leq d^{1-p} \int_{B_r(x_0)} |f| v^{-\gamma} \psi^{p} \, dx 
\leq d^{1-p} \| f \|_{L^{\frac{q}{p}}(B_r(x_0))} \Big\| \psi v^{-\frac{\gamma}{p}} \Big\|^{p}_{L^{\frac{p q}{q - p}}(B_r(x_0))} \\
&\leq d^{1-p} \| f \|_{L^{\frac{q}{p}}(B_r(x_0))} \Big\| \psi v^{-\frac{\gamma}{p}} \Big\|^{\frac{p(q - n)}{q}}_{L^{p}(B_r(x_0))} \Big\| \psi v^{-\frac{\gamma}{p}} \Big\|^{\frac{n p}{q}}_{L^{p \kappa}(B_r(x_0))} \\
&\leq d^{1-p} \| f \|_{L^{\frac{q}{p}}(B_r(x_0))} \left( \frac{n \delta}{q} \Big\| \psi v^{-\frac{\gamma}{p}} \Big\|^{p}_{L^{p \kappa}(B_r(x_0))} + \frac{q - n}{q} \delta^{-\frac{n}{q - n}} \Big\| \psi v^{-\frac{\gamma}{p}} \Big\|^{p}_{L^{p}(B_r(x_0))} \right),
\end{split}
\end{equation}
where the second inequality follows from H\"older's inequality. Moreover, since $q > n$, the third and fourth inequalities follow by interpolation and Young's inequality, respectively.

Therefore, if $\epsilon > p - 1$, then combining the estimates \eqref{estI}, \eqref{estJ} and \eqref{estK} into \eqref{IJK}, the desired inequality \eqref{negeng1} follows. Further, if $\epsilon \in (0, p - 1)$, then combining the estimates \eqref{estI}, \eqref{estJnew} and \eqref{estK} into \eqref{IJK}, the desired inequality \eqref{negeng2} follows. This completes the proof.
\end{proof}
Next, we establish the following logarithmic estimate for weak supersolutions of \eqref{meqn}.

\begin{Lemma}\label{loglem} (Logarithmic estimate)  
Suppose that $0 < r \leq 1$ with $r < \frac{R}{2}$. Let $u$ be a weak supersolution of \eqref{meqn} such that $u \geq 0$ in $B_R(x_0) \subset \Omega$. For $d > 0$, set $v = u + d$. Then there exists a positive constant $c = c(n,p,s,C_1,C_2,\Lambda)$ such that
\begin{equation}\label{loglemeqn}
\begin{split}
\int_{B_{\frac{3r}{2}}(x_0)} |\nabla \log v|^p \psi^p \, dx &\leq c r^n \bigl(r^{-p} + d^{1-p} R^{-p} \mathrm{Tail}(u_-; x_0, R)^{p-1} \bigr) \\
&\quad + c d^{1-p} R^{\frac{n(q-p)}{q}} \|f\|_{L^{\frac{q}{p}}(B_R(x_0))},
\end{split}
\end{equation}
for every $\psi \in C_c^\infty(B_{\frac{3r}{2}}(x_0))$ satisfying
\[
0 \leq \psi \leq 1 \quad \text{in } B_{\frac{3r}{2}}(x_0) \quad \text{and} \quad |\nabla \psi| \leq \frac{c}{r} \quad \text{in } B_{\frac{3r}{2}}(x_0),
\]
for some positive constant $c = c(n,p)$.
\end{Lemma}

\begin{proof}
Choosing the test function $\phi = v^{1-p} \psi^p$ in \eqref{weaksubsupsoln}, we obtain
\begin{equation}\label{logesttest}
0 \leq I + J + K,
\end{equation}
where
\[
I = \int_{B_{2r}(x_0)} \mathcal{A}(x, \nabla v) \nabla \bigl( v^{1-p} \psi^p \bigr) \, dx,
\]
\[
J = \int_{\mathbb{R}^n} \int_{\mathbb{R}^n} J_p(v(x) - v(y)) \bigl( v^{1-p}(x) \psi^p(x) - v^{1-p}(y) \psi^p(y) \bigr) \, d\mu,
\]
and
\[
K = - \int_{B_{2r}(x_0)} f v^{1-p} \psi^p \, dx.
\]

\textbf{Estimate of $I$:} \\
Using the property \eqref{A} of $\mathcal{A}$ along with Young's inequality, there exist constants $c_i = c_i(n,p,C_1,C_2) > 0$, $i=1,2$, such that
\begin{equation}\label{logI3}
I \leq - c_1 \int_{B_{\frac{3r}{2}}(x_0)} |\nabla \log v|^p \psi^p \, dx + c_2 r^{n-p}.
\end{equation}

\textbf{Estimate of $J$:} \\
By adapting the arguments from \cite[Pages 1288--1290, Lemma 1.3]{DKPahp} and using the properties of $\psi$ together with $0 < r \leq 1$ and $2r < R$, there exists a constant $c = c(n,p,s,\Lambda) > 0$ such that
\begin{equation}\label{logI1}
\begin{aligned}
J &\leq - \frac{1}{c} \int_{B_{2r}(x_0)} \int_{B_{2r}(x_0)} \left| \log\left(\frac{v(x)}{v(y)}\right) \right|^p \psi(y)^p \, d\mu \\
&\quad + c d^{1-p} r^n R^{-p} \mathrm{Tail}(u_-; x_0, R)^{p-1} + c r^{n-p}.
\end{aligned}
\end{equation}

\textbf{Estimate of $K$:} \\
Using the fact that $2r < R$ and H\"older's inequality, we have
\begin{equation}\label{I}
|K| = \left| \int_{B_{2r}(x_0)} f (u + d)^{1-p} \psi^p \, dx \right| \leq c(n,p) d^{1-p} R^{\frac{n(q-p)}{q}} \|f\|_{L^{\frac{q}{p}}(B_R(x_0))}.
\end{equation}

Combining \eqref{logI3}, \eqref{logI1} and \eqref{I} into \eqref{logesttest}, we obtain the desired estimate \eqref{loglemeqn}.
\end{proof}

\begin{Corollary}\label{logapp1}
If $d$ satisfies \eqref{d}, and in addition $\psi \equiv 1$ in $B_r(x_0)$ in Lemma \ref{loglem}, then from \eqref{loglemeqn} we deduce the estimate
\begin{equation}\label{logappeqn}
\int_{B_r(x_0)} |\nabla \log(u + d)|^p \, dx \leq C r^{n - p},
\end{equation}
for some constant $C = C(n,p,q,s,C_1,C_2,\Lambda) > 0$.
\end{Corollary}

\subsection{Energy estimate for subsolutions}
Next, we establish the following energy estimate for weak subsolutions of \eqref{meqn}.
\begin{Lemma}\label{subeng}
Let $u$ be a weak subsolution of \eqref{meqn} and denote $w = (u-k)_+$ with {$k\in\mathbb{R}$}. Then, for every $\delta > 0$, there exists a positive constant $c = c(p,C_1,C_2,\Lambda)$ such that
\begin{equation}\label{subengest}
\begin{split}
&\int_{B_r(x_0)} \psi^p \, |\nabla w|^p \, dx \\
&\leq c \Bigg\{ 
\int_{B_r(x_0)} w^p \, |\nabla \psi|^p \, dx
+ \int_{B_r(x_0)} \int_{B_r(x_0)} \max\{ w(x), w(y) \}^p \, |\psi(x)-\psi(y)|^p \, d\mu \\
&\quad + \sup_{x \in \mathrm{supp}\,\psi} \int_{\mathbb{R}^n \setminus B_r(x_0)} \frac{w^{p-1}(y)}{|x-y|^{n+ps}} \, dy 
\times \int_{B_r(x_0)} w\,\psi^p \, dx + {\int_{B_r(x_0)}|f|\,w\,\psi^p\,dx}
\Bigg\},
\end{split}
\end{equation}
whenever $B_r(x_0) \subset \Omega$ and $\psi \in C_c^\infty(B_r(x_0))$ is a nonnegative function.
\end{Lemma}

\begin{proof}
The proof follows by arguing similarly as in the proof of \cite[Lemma 3.1]{GKtams}, {by choosing $\varphi = w \psi^p$ as a test function in \eqref{weaksubsupsoln}}.
\end{proof}

\section{Supersolution and tail estimates}
\subsection{Inverse estimate for supersolutions}
\begin{Lemma}\label{Cor}
Suppose $0 < r \leq 1$ such that $r < R$. Let $u$ be a weak supersolution of \eqref{meqn} satisfying
$u \geq 0$ in $B_R(x_0) \subset \Omega$. Define $v = u + d$, where
\begin{equation}\label{d}
d > r^{p'} R^{-p'} \mathrm{Tail}(u_-; x_0, R) + r^{\frac{p-n}{p-1}} R^{\frac{n(q-p)}{q(p-1)}} \|f\|_{L^{\frac{q}{p}}(B_R(x_0))}^{\frac{1}{p-1}}.
\end{equation}
Let $0 < \gamma < 1$. Then there exist constants $C = C(n, p, q, s, \gamma, C_1, C_2, \La) > 0$ and $\theta = \theta(n, p, q, s) > 0$ such that
\begin{equation}\label{invest}
\sup_{B_{\sigma' r}(x_0)} (u + d)^{-1} \leq \left( \frac{C}{(\sigma - \sigma')^\theta} \right)^{\frac{1}{t}} \left( \fint_{B_{\sigma r}(x_0)} (u + d)^{-t} \, dx \right)^{\frac{1}{t}},
\end{equation}
holds for all $0 < \gamma \leq \sigma' < \sigma \leq 1$ and for all $t > 0$.
\end{Lemma}

\begin{proof}
Let $0 < \gamma < \sigma' < \sigma \leq 1$. Recall $\kappa = \frac{n}{n-p}$, where $1 < p < n$. For $j = 0, 1, \ldots,$ define
\[
\sigma_j = \sigma - (\sigma - \sigma') \big(1 - \kappa^{-j}\big), \quad r_j = \sigma_j r, \quad \text{and} \quad B_j = B_{r_j}(x_0).
\]
Let $\psi_j \in C_c^\infty(B_j)$ satisfy
\begin{equation}\label{psi}
\begin{cases}
0 \leq \psi_j \leq 1 \text{ in } B_j, \\[6pt]
|\nabla \psi_j| \leq \dfrac{c \kappa^j}{(\sigma - \sigma') r} \text{ in } B_j, \\[6pt]
\psi_j \equiv 1 \text{ in } B_{j+1}, \text{ and } \mathrm{dist}(\mathrm{supp} \, \psi_j, \mathbb{R}^n \setminus B_j) \geq 2^{-j-1} r,
\end{cases}
\end{equation}
for some constant $c=c(n,p)>0$.
Let $\epsilon > p - 1$. For $\alpha = \epsilon - p + 1 > 0$, we observe that
\begin{equation}\label{11}
\fint_{B_{j+1}} v^{-\alpha \kappa} \, dx \leq \frac{1}{|B_{j+1}|} \int_{B_j} \left| v^{-\frac{\alpha}{p}} \psi_j \right|^{p \kappa} \, dx.
\end{equation}
By Lemma \ref{emb}, we have
\begin{equation}\label{22}
\begin{split}
\int_{B_j} \left| v^{-\frac{\alpha}{p}} \psi_j \right|^{p \kappa} \, dx &\leq C \left( \int_{B_j} \left| \nabla \big( v^{-\frac{\alpha}{p}} \psi_j \big) \right|^p \, dx \right)^\kappa \\
&\leq C \left( \int_{B_j} \left| \nabla \big( v^{-\frac{\alpha}{p}} \big) \right|^p \psi_j^p \, dx \right)^\kappa + C \left( \int_{B_j} |\nabla \psi_j|^p v^{-\alpha} \, dx \right)^\kappa := (I + J),
\end{split}
\end{equation}
where $C = C(n,p) > 0$ is a constant.

\textbf{Estimate of $I$:} Using the estimate \eqref{negeng1} from Lemma \ref{inveng}, {for every $\delta>0$}, we get
\begin{equation}\label{estIInew}
\begin{split}
& \int_{B_j} \psi_j^p \left| \nabla \left( v^{-\frac{\alpha}{p}} \right) \right|^p \, dx \\
&\leq C \alpha^p \Bigg\{ \int_{B_j} v^{-\alpha} |\nabla \psi_j|^p \, dx + \int_{B_j} \int_{B_j} \big( v^{-\alpha}(x) + v^{-\alpha}(y) \big) |\psi_j(x) - \psi_j(y)|^p \, d\mu \\
&\quad + \left( \sup_{x \in \mathrm{supp} \, \psi_j} \int_{\mathbb{R}^n \setminus B_j} \frac{dy}{|x - y|^{n + p s}} + d^{1-p} \int_{\mathbb{R}^n \setminus B_R(x_0)} \frac{u_-(y)^{p-1}}{|y - x_0|^{n + p s}} \, dy \right) \int_{B_j} v^{-\alpha} \psi_j^p \, dx \\
&\quad + d^{1-p} \|f\|_{L^{\frac{q}{p}}(B_R(x_0))} \Big( \delta \|\psi_j v^{-\frac{\alpha}{p}}\|_{L^{p \kappa}(B_j)}^p + \delta^{-\frac{n}{q - n}} \|\psi_j v^{-\frac{\alpha}{p}}\|_{L^p(B_j)}^p \Big) \Bigg\} \\
&:= C \alpha^p (I_1 + I_2 + I_3 + I_4),
\end{split}
\end{equation}
where $C = C(C_1, C_2, \Lambda, p, \epsilon) > 0$ is a constant bounded as long as $\epsilon$ stays away from $p - 1$.

\textbf{Estimate of $I_1$:} Using condition \eqref{psi} on $\psi_j$, we have
\begin{equation}\label{estI1}
I_1 \leq C \left( \frac{\kappa^j}{(\sigma - \sigma') r_j} \right)^p \int_{B_j} v^{-\alpha} \, dx.
\end{equation}

\textbf{Estimate of $I_2$:} Again, by condition \eqref{psi} on $\psi_j$ and since $0 < r \leq 1$, it follows that
\begin{equation}\label{estI2}
I_2 \leq C \left( \frac{\kappa^j}{(\sigma - \sigma') r_j} \right)^p \int_{B_j} v^{-\alpha} \, dx.
\end{equation}

\textbf{Estimate of $I_3$:} For any $x \in \mathrm{supp} \, \psi_j$ and $y \in \mathbb{R}^n \setminus B_j$, using the property \eqref{psi}, we have
\[
\frac{1}{|x - y|} = \frac{1}{|y - x_0|} \frac{|y - x_0|}{|x - y|} \leq \frac{1}{|y - x_0|} \left( 1 + \frac{|x - x_0|}{|x - y|} \right) \leq \frac{1}{|y - x_0|} \left(1 + \frac{r}{2^{-j-1} r} \right) \leq \frac{2^{j+2}}{|y - x_0|}.
\]
Using $0 < r \leq 1$, this implies
\[
\sup_{x \in \mathrm{supp} \, \psi_j} \int_{\mathbb{R}^n \setminus B_j} \frac{dy}{|x - y|^{n + p s}} \leq C \frac{2^{j(n + p s)}}{(\sigma - \sigma')^p r_j^p}.
\]
Again, using $0 < r \leq 1$ along with the assumption \eqref{d} on $d$, we get
\[
d^{1-p} \int_{\mathbb{R}^n \setminus B_R(x_0)} \frac{u_-(y)^{p-1}}{|x-y|^{n + p s}} \, dy \leq \leq C \frac{2^{j(n + p s)}}{(\sigma - \sigma')^p r_j^p}.
\]
Hence,
\begin{equation}\label{estI3}
I_3 \leq C \frac{2^{j(n + p s)}}{(\sigma - \sigma')^p r_j^p} \int_{B_j} v^{-\alpha} \, dx.
\end{equation}

\textbf{Estimate of $I_4$:} Set $\delta = \delta_0 r^{q-n} R^{-\frac{(q-p)(q-n)}{q}}$ where $\delta_0 > 0$ is a constant to be chosen later. Using the assumption \eqref{d} on $d$ and $r < R$, we have
\begin{equation}\label{estI4}
I_4 \leq \delta_0 \|\psi_j v^{-\frac{\alpha}{p}}\|_{L^{p \kappa}(B_j)}^p + \delta_0^{-\frac{n}{q - n}} r^{-p} \|\psi_j v^{-\frac{\alpha}{p}}\|_{L^p(B_j)}^p.
\end{equation}
Combining the estimates \eqref{estI1}, \eqref{estI2}, \eqref{estI3}, and \eqref{estI4} into \eqref{estIInew}, we obtain
\begin{equation}\label{estIII}
\begin{split}
I &\leq C \alpha^{p\kappa} \sum_{i=1}^4 I_i^\kappa \\
&= C \alpha^{p\kappa} \left( \frac{(2\kappa)^{j(n+ps+p)}}{(\sigma-\sigma')^{p} r_j^{p}} \int_{B_j} v^{-\alpha} \, dx \right)^\kappa + C \alpha^{p\kappa} \left( \delta_0^{\kappa} \big\| \psi_j v^{-\frac{\alpha}{p}} \big\|_{L^{p\kappa}(B_j)}^{p\kappa} + \delta_0^{-\frac{n\kappa}{q-n}} r^{-p\kappa} \big\| \psi_j v^{-\frac{\alpha}{p}} \big\|_{L^p(B_j)}^{p\kappa} \right).
\end{split}
\end{equation}
In the above estimates $C = C(n,p,s,C_1,C_2,\La,\epsilon,\gamma) > 0$ is a constant bounded when $\epsilon$ is bounded away from $p-1$. 

Now, choosing $\delta_0 > 0$ such that $C \alpha^{p\kappa} \delta_0^{\kappa} = \frac{1}{2}$, the above inequality simplifies to
\begin{equation}\label{estIV}
\begin{split}
I &\leq C \alpha^{p\kappa} \left( \frac{(2\kappa)^{j(n+ps+p)}}{(\sigma-\sigma')^{p} r_j^{p}} \int_{B_j} v^{-\alpha} \, dx \right)^\kappa + \left( \frac{1}{2} \big\| \psi_j v^{-\frac{\alpha}{p}} \big\|_{L^{p\kappa}(B_j)}^{p\kappa} + 2^{\frac{n}{q-n}} C^{\frac{n}{q-n}} \alpha^{p\kappa \left(1 + \frac{n}{q-n}\right)} r^{-p\kappa} \big\| \psi_j v^{-\frac{\alpha}{p}} \big\|_{L^p(B_j)}^{p\kappa} \right) \\
&\leq C \max \left\{ \alpha^{p\kappa}, \alpha^{p\kappa \left(1 + \frac{n}{q-n}\right)} \right\} \left( \frac{(2\kappa)^{j(n+ps+p)}}{(\sigma-\sigma')^{p} r_j^{p}} \int_{B_j} v^{-\alpha} \, dx \right)^\kappa + \frac{1}{2} \big\| \psi_j v^{-\frac{\alpha}{p}} \big\|_{L^{p\kappa}(B_j)}^{p\kappa},
\end{split}
\end{equation}
where $C = C(n,p,s,C_1,C_2,\La,\epsilon,\gamma) > 0$ is a constant bounded as long as $\epsilon$ is bounded away from $p-1$.

\textbf{Estimate of $J$:} Using \eqref{psi}, we have
\begin{equation}\label{estJJ}
J \leq C \left( \frac{(2\kappa)^{j(n+ps+p)}}{(\sigma-\sigma')^{p} r_j^{p}} \int_{B_j} v^{-\alpha} \, dx \right)^\kappa,
\end{equation}
where $C = C(n,p,s,\gamma) > 0$ is a constant. 

Hence, combining \eqref{estIV} and \eqref{estJJ} in \eqref{22}, we deduce
\begin{equation}\label{3}
\int_{B_j} \left| v^{-\frac{\alpha}{p}} \psi_j \right|^{p\kappa} dx \leq C \left( \max \left\{ \alpha^{p\kappa}, \alpha^{p\kappa \left(1 + \frac{n}{q-n} \right)} \right\} + 1 \right) \left( \frac{(2\kappa)^{j(n+ps+p)}}{(\sigma-\sigma')^{p} r_j^{p}} \int_{B_j} v^{-\alpha} dx \right)^\kappa,
\end{equation}
where $C = C(n,p,s,C_1,C_2,\La,\epsilon,\gamma) > 0$ is bounded as long as $\epsilon$ is bounded away from $p-1$.

Using \eqref{3} in \eqref{11}, we obtain the estimate ({In the first line below, I have removed the typo $r_j^p$ from the denominator of the averaged integral.})
\begin{equation}\label{4}
\begin{split}
\fint_{B_{j+1}} v^{-\alpha \kappa} dx &\leq C \frac{r_j^{-p\kappa} |B_j|^\kappa}{|B_{j+1}|} \left( \frac{\left( \max \left\{ \alpha^{p}, \alpha^{p \left(1 + \frac{n}{q-n} \right)} \right\} + 1 \right) (2\kappa)^{j(n+ps+p)}}{(\sigma - \sigma')^{p}} \fint_{B_j} v^{-\alpha} dx \right)^\kappa \\
&\leq C \left( \frac{\left( \max \left\{ \alpha^{p}, \alpha^{p \left(1 + \frac{n}{q-n} \right)} \right\} + 1 \right) (2\kappa)^{j(n+ps+p)}}{(\sigma - \sigma')^{p}} \fint_{B_j} v^{-\alpha} dx \right)^\kappa,
\end{split}
\end{equation}
where we used that the quantity $\frac{r_j^{-p\kappa} |B_j|^\kappa}{|B_{j+1}|}$ is bounded by a constant independent of $j, r, \sigma$ and $\sigma'$. Also, $C = C(n,p,s,C_1,C_2,\La,\epsilon,\gamma) > 0$ is bounded as long as $\epsilon$ is bounded away from $p-1$.

We now apply the Moser iteration technique to \eqref{4}. To this end, choose $\alpha_j = p \kappa^j$ for $j=0,1,\dots$. Noting that $\alpha_j \geq p$, iterating inequality \eqref{4} yields
\begin{align*}
\left( \fint_{B_0} v^{-p} dx \right)^{\frac{1}{p}} &\geq \left( \frac{\sigma - \sigma'}{C} \right)^{1 + \kappa^{-1} + \cdots + \kappa^{-m+1}} \prod_{j=0}^{m-1} (2\kappa)^{-\frac{(l + n + ps + p) j}{p \kappa^j}} \left( \fint_{B_m} v^{-p \kappa^m} dx \right)^{\frac{1}{p \kappa^m}},
\end{align*}
where $l = \frac{p q}{q - n}$ and $C = C(n,p,q,s,C_1,C_2,\La,\gamma) > 0$ is a constant. Following the approach in \cite[page 15]{KK}, letting $m \to \infty$ in the above inequality, we conclude the desired estimate \eqref{invest}.
\end{proof}

\subsection{Reverse H\"older inequality for supersolutions}
\begin{Lemma}\label{RevH}
Let $0 < r \leq 1$ with $r < R$. Suppose $u$ is a weak supersolution of \eqref{meqn} such that $u \geq 0$ in $B_R(x_0) \subset \Om$. Define $v = u + d$, where $d$ satisfies \eqref{d}. For any $0 < \gamma < 1$, there exist constants 
\[
C = C(n,p,q,s,C_1,C_2,\La,\la,\gamma) > 0 \quad \text{and} \quad 
\theta = \theta(n,p,s) > 0
\] 
such that
\begin{equation}\label{revest}
\begin{split}
\Bigg( \fint_{B_{\sigma' r}(x_0)} v^{\lambda} \, dx \Bigg)^{\frac{1}{\lambda}}
\;\leq\; 
\Bigg( \frac{C}{(\sigma - \sigma')^{\theta}} \Bigg)^{\frac{1}{\mu}}
\Bigg( \fint_{B_{\sigma r}(x_0)} v^{\mu} \, dx \Bigg)^{\frac{1}{\mu}},
\end{split}
\end{equation}
for all $0 < \gamma \leq \sigma' < \sigma \leq 1$ and $0 < \mu < \lambda < \lambda_0$, where $\lambda_0 = (p-1)\kappa$.
\end{Lemma}

\begin{Remark}\label{Rmk}
It is important to note that Lemma \ref{RevH} is derived here using inequality \eqref{alg2}, which differs from the one employed to establish the reverse H\"older inequality for the equation \eqref{meqn1} in \cite[estimate (8.16)]{GKtams} when $f \equiv 0$. Moreover, the above inequality \eqref{revest} is new for nontrivial $f$.
\end{Remark}

\begin{proof}
We take $0 < \gamma < \sigma' < \sigma \leq 1$ and recall that $\kappa = \frac{n}{n-p}$ with $1 < p < n$. We partition the interval $(\sigma, \sigma')$ into $m$ subintervals, where $j = 0,1,\ldots,m$ (with $m$ to be chosen later), and define
\[
\sigma_j = \sigma - (\sigma - \sigma') \frac{1 - \kappa^{-j}}{1 - \kappa^{-m}}, 
\quad r_j = \sigma_j r, 
\quad B_j = B_{r_j}(x_0).
\]
Let $\psi_j \in C_c^\infty(B_j)$ be as defined in \eqref{psi}. Let $\epsilon\in(0,p-1)$, then for $\beta = p-1 - \epsilon > 0$, we obtain
\begin{equation}\label{1}
\fint_{B_{j+1}} v^{\beta \kappa} \, dx 
\;\leq\; 
\frac{1}{|B_{j+1}|} \int_{B_j} \big| v^{\frac{\beta}{p}} \psi_j \big|^{p\kappa} \, dx.
\end{equation}

Applying Lemma \ref{emb} gives
\begin{equation}\label{2}
\begin{split}
\int_{B_j} \big| v^{\frac{\beta}{p}} \psi_j \big|^{p\kappa} \, dx
&\leq C \Bigg( \int_{B_j} \big| \nabla( v^{\frac{\beta}{p}} \psi_j ) \big|^p \, dx \Bigg)^{\kappa} \\
&\leq C \Bigg( \int_{B_j} \big| \nabla(v^{\frac{\beta}{p}}) \big|^p \psi_j^p \, dx \Bigg)^{\kappa}
+ C \Bigg( \int_{B_j} |\nabla \psi_j|^p v^{\beta} \, dx \Bigg)^{\kappa} \\
&=: I + J,
\end{split}
\end{equation}
where $C = C(n,p) > 0$.

\medskip
\noindent\textbf{Estimate of $I$.}  
Using estimate \eqref{negeng2} from Lemma \ref{inveng}, we obtain
\begin{equation}\label{estII}
\begin{split}
\int_{B_j} \psi_j^p \Big| \nabla \big( v^{\frac{\beta}{p}} \big) \Big|^p \, dx
&\leq C \Bigg\{ \int_{B_j} v^{\beta} |\nabla \psi_j|^p \, dx 
+ \int_{B_j}\int_{B_j} (v^{\beta}(x) + v^{\beta}(y)) |\psi_j(x)-\psi_j(y)|^p \, d\mu \\
&\quad + \Bigg( \sup_{x \in \mathrm{supp}\,\psi_j} \int_{\mathbb{R}^n \setminus B_j} \frac{dy}{|x-y|^{n+ps}}
+ d^{1-p} \int_{\mathbb{R}^n \setminus B_R(x_0)} \frac{u_-(y)^{p-1}}{|y-x_0|^{n+ps}} \, dy \Bigg) 
\int_{B_j} v^{\beta} \psi_j^p \, dx \\
&\quad + d^{1-p} \|f\|_{L^{q/p}(B_R(x_0))} \Big( \delta \| \psi v^{\frac{\beta}{p}} \|_{L^{p\kappa}(B_j)}^p 
+ \delta^{-\frac{n}{q-n}} \| \psi_j v^{-\frac{\alpha}{p}} \|_{L^p(B_j)}^p \Big) \Bigg\},
\end{split}
\end{equation}
where $C = C(n,p,s,C_1,C_2,\La,\beta) > 0$, bounded provided $\beta$ remains away from $p-1$.

Proceeding exactly as in the proof of Lemma \ref{Cor}, we arrive at
\begin{equation}\label{5}
\begin{split}
\Bigg( \fint_{B_{j+1}} v^{\beta \kappa} \, dx \Bigg)^{\frac{1}{\beta \kappa}}
&\leq 
\Bigg( \frac{C (2\kappa)^{j(n+ps+p)}}{(\sigma - \sigma')^p} 
\fint_{B_j} v^{\beta} \, dx \Bigg)^{\frac{1}{\beta}},
\end{split}
\end{equation}
using the fact that 
$\frac{r_j^{-p\kappa} |B_j|^{\kappa}}{|B_{j+1}|}$ is bounded independently of $j, r, \sigma, \sigma'$.  
Here $C = C(n,p,s,C_1,C_2,\La,\beta) > 1$ is bounded if $\beta$ stays away from $p-1$.

\medskip  
To apply Moser iteration to \eqref{5}, we fix $\lambda > \mu$ and choose $m$ such that 
$\mu \kappa^{m-1} \leq \lambda \leq \mu \kappa^m$.  
Let $\rho_0 \leq \mu$ with $\lambda = \kappa^m \rho_0$, and set $\rho_j = \kappa^j \rho_0$ for $j=0,1,\ldots,m$.  
Then, iterating \eqref{5}, we obtain
\begin{equation}\label{6}
\begin{split}
\Bigg( \fint_{B_m} v^{\lambda} \, dx \Bigg)^{\frac{1}{\lambda}}
&\leq \Bigg( \frac{C (2\kappa)^m}{\sigma - \sigma'} \Bigg)^{\frac{n+ps+p}{\rho_{m-1}}}
\Bigg( \fint_{B_{m-1}} v^{\rho_{m-1}} \, dx \Bigg)^{\frac{1}{\rho_{m-1}}} \\
&\leq \vdots \\
&\leq \Bigg( \frac{C_{\mathrm{prod}}(m)}{(\sigma - \sigma')^{\kappa^*}}
\fint_{B_{\sigma r}(x_0)} v^{\rho_0} \, dx \Bigg)^{\frac{1}{\rho_0}},
\end{split}
\end{equation}
where
\[
C_{\mathrm{prod}}(m) = C^{\kappa^*} \prod_{j=0}^{m-1} (2\kappa)^{\frac{(n+ps+p)(j+1)}{\kappa^j}},
\qquad 
\kappa^* = (n+ps+p) \sum_{j=0}^{m-1} \kappa^{-j}
= \frac{(n+ps+p)\kappa (1 - \kappa^{-m})}{\kappa - 1}.
\]

Note that the constant $C$ depends on $\lambda$ due to the singularity of the constant in Lemma \ref{inveng} at $\epsilon=0$.  
Since $C_{\mathrm{prod}}(m)$ is uniformly bounded in $m$, applying H\"older’s inequality to \eqref{6} yields
\begin{equation}\label{66}
\begin{split}
\Bigg( \fint_{B_{\sigma' r}(x_0)} v^{\lambda} \, dx \Bigg)^{\frac{1}{\lambda}}
&\leq 
\Bigg( \frac{C}{(\sigma - \sigma')^{\kappa^*}} \Bigg)^{\frac{1}{\rho_0}}
\Bigg( \fint_{B_{\sigma r}(x_0)} v^{\mu} \, dx \Bigg)^{\frac{1}{\mu}}.
\end{split}
\end{equation}

Finally, since $\mu \kappa^{m-1} \leq \rho_0 \kappa^m$, we deduce $\rho_0 \geq \frac{\mu}{\kappa}$, and hence the desired estimate \eqref{revest} follows with
\[
\theta = \frac{(n+ps+p)\kappa^2}{\kappa - 1}.
\]
\end{proof}

\subsection{Logarithmic estimate for supersolutions}
\begin{Lemma}\label{logest2}
Suppose $0 < r \leq 1$ with $r < \frac{R}{2}$. Let $u$ be a weak supersolution of \eqref{meqn} such that $u \geq 0$ in $B_R(x_0) \subset \Omega$. Define $v = u + d$, where $d$ satisfies \eqref{d}. Let $\psi \in C_c^\infty(B_{\frac{3r}{2}}(x_0))$ be such that
\[
0 \leq \psi \leq 1, 
\qquad 
|\nabla \psi| \leq \frac{c}{r} \ \text{in } B_{\frac{3r}{2}}(x_0),
\qquad 
\psi \equiv 1 \ \text{in } B_r(x_0),
\]
for some constant $c = c(n,p) > 0$. Then there exists a constant $c = c(n,p,q,s,C_1,C_2,\La) > 0$ such that
\[
\mu_1 := \big| \{x \in B_r(x_0) : \log v + b > a\} \big|
\;\leq\; C \frac{|B_{\frac{r}{2}}(x_0)|}{a^p}, 
\quad \forall a > 0,
\]
and
\[
\mu_2 := \big| \{x \in B_r(x_0) : -\log v - b > a\} \big|
\;\leq\; C \frac{|B_{\frac{r}{2}}(x_0)|}{a^p}, 
\quad \forall a > 0,
\]
where
\begin{equation}\label{b}
b = - (\log v)_{\psi,B_{\frac{3r}{2}}(x_0)}
= - \frac{\int_{B_{\frac{3r}{2}}(x_0)} \psi^p \log v \, dx}
{\int_{B_{\frac{3r}{2}}(x_0)} \psi^p \, dx}.
\end{equation}
\end{Lemma}

\begin{proof}
Let $f = \log v + b$. By applying the weighted Poincar\'e inequality in Lemma \ref{wgtPoin} with $\phi = \psi^p$, there exists a constant $C = C(n,p,q,s,C_1,C_2,\La) > 0$ such that
\begin{equation}\label{7}
\int_{B_r(x_0)} |f|^p \, dx 
\;\leq\; \int_{B_{\frac{3r}{2}}(x_0)} |f|^p \psi^p \, dx
\;\leq\; C r^p \int_{B_{\frac{3r}{2}}(x_0)} |\nabla \log v|^p \psi^p \, dx 
\;\leq\; C r^n,
\end{equation}
where the last inequality follows from Lemma \ref{loglem} together with the assumption on $d$.  

Using \eqref{7} and applying Chebyshev’s inequality, for any $a > 0$ we obtain
\[
\mu_1 + \mu_2 \;\leq\; \frac{2}{a^p} \int_{B_r(x_0)} |f|^p \, dx 
\;\leq\; C \frac{|B_{\frac{r}{2}}(x_0)|}{a^p},
\]
for some constant $C = C(n,p,q,s,C_1,C_2,\La) > 0$. This proves the lemma.
\end{proof}
\subsection{Tail estimate}
\begin{Lemma}\label{tailest}
Let $u$ be a weak solution of \eqref{meqn} such that $u \geq 0$ in $B_R(x_0) \subset \Omega$. Then there exists a positive constant $c = c(n,p,q,s,C_1,C_2,\Lambda)$ such that  
$$
\mathrm{Tail}(u_+; x_0, r) \leq c \sup_{B_r(x_0)} u 
+ c r^{p'} R^{-p'} \, \mathrm{Tail}(u_-; x_0, R) 
+ c r^{\frac{p-n}{p-1}} R^{\frac{n(q-p)}{q(p-1)}} \, \| f \|_{L^{\frac{q}{p}}(B_R(x_0))}^{\frac{1}{p-1}},
$$
whenever $0 < r < R$ with $r \in (0,1]$.
\end{Lemma}

\begin{proof}
Let $M = \sup_{B_r(x_0)} u$ and let $\psi \in C_c^\infty(B_r(x_0))$ be such that $0 \leq \psi \leq 1$ in $B_r(x_0)$, with $\psi \equiv 1$ in $B_{\frac{r}{2}}(x_0)$ and $|\nabla \psi| \leq \frac{8}{r}$ in $B_r(x_0)$. Setting $w = u - 2M$ and choosing $\varphi = w \psi^p$ as a test function in \eqref{weaksubsupsoln}, we obtain
\begin{equation}\label{tailtst}
I_1 + I_2 + I_3 - I_4 = 0,
\end{equation}
where
$$
I_1 = \int_{B_r(x_0)} \mathcal{A}(x,\nabla u) \nabla (w \psi^p) \, dx,
$$

$$
I_2 = \int_{B_r(x_0)} \int_{B_r(x_0)} J_p(u(x)-u(y)) \big(w(x)\psi^p(x) - w(y)\psi^p(y)\big) \, d\mu,
$$

$$
I_3 = 2 \int_{B_r(x_0)} \int_{\mathbb{R}^n \setminus B_r(x_0)} J_p(u(x)-u(y)) w(x) \psi^p(x) \, d\mu, 
\quad \text{and} \quad
I_4 = \int_{B_r(x_0)} f w \psi^p \, dx.
$$

Taking into account that $r\in(0,1]$ and proceeding exactly as in the estimates of $I_i$, $i=1,2,3$, in \cite[page 5410]{GKtams}, we obtain
$$
I_1 + I_2 + I_3 \geq c M r^{-p} \, \mathrm{Tail}(u_+; x_0, r)^{p-1} |B_r(x_0)|
- c M R^{-p} \, \mathrm{Tail}(u_-; x_0, R)^{p-1} |B_r(x_0)|
- c M^p r^{-p} |B_r(x_0)|,
$$
for some positive constant $c = c(n,p,s,C_1,C_2,\Lambda)$. Moreover, we observe that
$$
I_4 \leq 3M \, {\| f \|_{L^{\frac{q}{p}}(B_R(x_0))}} \, |B_r(x_0)|^{\frac{q-p}{q}}.
$$

Combining the above estimates in \eqref{tailtst}, and using that $0<r<R$, the result follows.
\end{proof}

\section{Proof of the main results}
\subsection{Proof of Theorem \ref{mthm} (Using John-Nirenberg lemma)}
We establish the result in three steps.

\medskip
\noindent\textbf{Step 1.}  
We claim that there exist constants $\epsilon = \epsilon(n,p,q,s,C_1,C_2,\Lambda) > 0$ and $c = c(n,p,q,s,C_1,C_2,\Lambda,\epsilon) > 0$ such that
\begin{equation}\label{neg-poseqn}
\Bigg( \fint_{B_r(x_0)} v^{\epsilon} \, dx \Bigg)^{\frac{1}{\epsilon}}
\;\leq\; 
c \Bigg( \fint_{B_r(x_0)} v^{-\epsilon} \, dx \Bigg)^{-\frac{1}{\epsilon}},
\end{equation}
where $v = u + d$ with $d$ satisfying \eqref{d}.  

To prove this, set $w = \log v$. By the classical Poincar\'e inequality,
\begin{equation}\label{Poinapp1}
\int_{B_r(x_0)} |w(x) - (w)_{1,B_r(x_0)}|^p \, dx
\;\leq\; c_1 r^p \int_{B_r(x_0)} |\nabla w|^p \, dx,
\end{equation}
for some constant $c_1 = c_1(n,p) > 0$. Using H\"older’s inequality together with \eqref{logappeqn} and \eqref{Poinapp1}, we deduce
\[
\fint_{B_r(x_0)} |w(x) - (w)_{1,B_r(x_0)}| \, dx \;\leq\; K,
\]
where $K = K(n,p,q,s,C_1,C_2,\Lambda) > 0$. By Theorem \ref{JN-lem}, there exists a constant $\nu = \nu(n) > 0$ such that
\[
\fint_{B_r(x_0)} e^{\nu |w(x) - (w)_{1,B_r(x_0)}|} \, dx \;\leq\; 2.
\]
Hence,
\[
\fint_{B_r(x_0)} e^{\frac{\nu w}{K}} \, dx 
\cdot \fint_{B_r(x_0)} e^{-\frac{\nu w}{K}} \, dx 
\;\leq\; 4.
\]
This yields \eqref{neg-poseqn} with $\epsilon = \frac{\nu}{K}$ and $c = 4^{1/\epsilon}$.

\medskip
\noindent\textbf{Step 2.}  
Choosing $\sigma = 1$, $\sigma' = \frac{1}{2}$, and $t = \epsilon$ in Lemma \ref{Cor}, and combining with \eqref{neg-poseqn}, we obtain
\begin{equation}\label{pwh1}
\Bigg( \fint_{B_r(x_0)} v^{\epsilon} \, dx \Bigg)^{\frac{1}{\epsilon}}
\;\leq\; c \Big( \inf_{B_{\frac{r}{2}}(x_0)} u + d \Big),
\end{equation}
for some constant $c = c(n,p,q,s,C_1,C_2,\La,\epsilon) > 0$.

\medskip
\noindent\textbf{Step 3.}  
Let $\eta \in (0,(p-1)\kappa)$.  

If $\eta < \epsilon$, applying H\"older’s inequality to \eqref{pwh1} gives
\begin{equation}\label{pwh2}
\Bigg( \fint_{B_r(x_0)} u^{\eta} \, dx \Bigg)^{\frac{1}{\eta}}
\;\leq\; c \Big( \inf_{B_{\frac{r}{2}}(x_0)} u + d \Big),
\end{equation}
where $c = c(n,p,q,s,C_1,C_2,\La) > 0$.

If $\eta > \epsilon$, then choosing $\mu = \epsilon$, $\lambda = \eta$, $\sigma = 1$, $\sigma' = \frac{1}{2}$ in Lemma \ref{RevH}, and combining with \eqref{pwh1}, we again obtain
\begin{equation}\label{fwh}
\Bigg( \fint_{B_r(x_0)} u^{\eta} \, dx \Bigg)^{\frac{1}{\eta}}
\;\leq\; c \Big( \inf_{B_{\frac{r}{2}}(x_0)} u + d \Big),
\end{equation}
where $c = c(n,p,q,s,C_1,C_2,\La,\eta) > 0$.

Now, for any $\delta > 0$, we choose
\[
d = r^{p'} R^{-p'} \mathrm{Tail}(u_-; x_0,R) 
+ r^{\frac{p-n}{p-1}} R^{\frac{n(q-p)}{q(p-1)}} 
\|f\|_{L^{\frac{q}{p}}(B_R(x_0))}^{\frac{1}{p-1}} + \delta.
\]
Then from \eqref{pwh2} and \eqref{fwh}, for any $\delta > 0$,
\begin{equation}\label{pwh3}
\Bigg( \fint_{B_{\frac{r}{2}}(x_0)} u^{\eta} \, dx \Bigg)^{\frac{1}{\eta}}
\;\leq\; c \Big( \inf_{B_{\frac{r}{2}}(x_0)} u 
+ r^{p'} R^{-p'} \mathrm{Tail}(u_-; x_0,R)
+ r^{\frac{p-n}{p-1}} R^{\frac{n(q-p)}{q(p-1)}} 
\|f\|_{L^{\frac{q}{p}}(B_R(x_0))}^{\frac{1}{p-1}} + \delta \Big),
\end{equation}
where $c = c(n,p,q,s,C_1,C_2,\La,\eta) > 0$. Letting $\delta \to 0$ in \eqref{pwh3} gives the desired estimate \eqref{mthmest}. This completes the proof.

\subsection{Proof of Theorem \ref{mthm} (Using Bombieri--Giusti lemma)}
Let $v = u + d$, where $d$ satisfies \eqref{d}. Assume $\frac{1}{2} \leq \sigma' < \sigma \leq 1$. Define $U_m = B_{mr}(x_0)$ for every $m \in [\frac{1}{2},1]$. Set
\[
w_1 = e^{-b} v^{-1}, 
\qquad 
w_2 = e^b v,
\]
where
\[
b = - \frac{\int_{B_{\frac{3r}{2}}(x_0)} \psi^p \log v \, dx}
{\int_{B_{\frac{3r}{2}}(x_0)} \psi^p \, dx},
\]
as in \eqref{b}.  

By Lemma \ref{Cor}, there exist constants $C = C(n,p,q,s,C_1,C_2,\La) > 0$ and $\theta = \theta(n,p,q,s) > 0$ such that
\begin{equation}\label{invap1}
\sup_{U_{\sigma'}} w_1 
\;\leq\; \Bigg( \frac{C}{(\sigma - \sigma')^{\theta}} \fint_{U_\sigma} w_1^t \, dx \Bigg)^{\frac{1}{t}}
\end{equation}
for every $t > 0$. Moreover, by Lemma \ref{logest2}, there exists a constant $C = C(n,p,q,s,C_1,C_2,\La) > 0$ such that
\begin{equation}\label{logestap1}
\big| \{ x \in U_1 : \log w_1 > a \} \big|
\;\leq\; C \frac{|U_{\frac{1}{2}}|}{a^p}, 
\quad \forall a > 0.
\end{equation}
Hence, by Lemma \ref{Bom-Giu},
\begin{equation}\label{BGap1}
\sup_{U_{\frac{1}{2}}} w_1 \;\leq\; C,
\end{equation}
for some $C = C(n,p,q,s,C_1,C_2,\La) > 0$.  

Similarly, by Lemma \ref{RevH}, there exist constants $C = C(n,p,q,s,C_1,C_2,\La,\eta) > 0$ and $\theta = \theta(n,p,s) > 0$ such that
\begin{equation}\label{revHap1}
\Bigg( \fint_{U_{\sigma'}} w_2^\eta \, dx \Bigg)^{\frac{1}{\eta}}
\;\leq\; 
\Bigg( \frac{C}{(\sigma - \sigma')^{\theta}} \fint_{U_\sigma} w_2^\mu \, dx \Bigg)^{\frac{1}{\mu}},
\end{equation}
for every $0 < \mu < \eta < (p-1)\kappa$. In addition, by Lemma \ref{logest2}, there exists a constant $C = C(n,p,q,s,C_1,C_2,\La) > 0$ such that
\begin{equation}\label{logestap2}
\big| \{ x \in U_1 : \log w_2 > a \} \big|
\;\leq\; C \frac{|U_{\frac{1}{2}}|}{a^p}, 
\quad \forall a > 0.
\end{equation}
Thus, by Lemma \ref{Bom-Giu},
\begin{equation}\label{BGap2}
\Bigg( \fint_{U_{\frac{1}{2}}} w_2^\eta \, dx \Bigg)^{\frac{1}{\eta}}
\;\leq\; C,
\end{equation}
for some $C = C(n,p,q,s,C_1,C_2,\La,\eta) > 0$.  

Now, for any $\delta > 0$, we choose
\[
d = r^{p'} R^{-p'} \mathrm{Tail}(u_-; x_0,R) 
+ r^{\frac{p-n}{p-1}} R^{\frac{n(q-p)}{q(p-1)}} 
\|f\|_{L^{\frac{q}{p}}(B_R(x_0))}^{\frac{1}{p-1}} + \delta.
\]
Multiplying \eqref{BGap1} and \eqref{BGap2} and then letting $\delta \to 0$ yields the desired result \eqref{mthmest}. This completes the proof.

\subsection{Proof of Theorem \ref{mthmbdd}}
Taking into account Lemma \ref{subeng}, the proof follows the lines of the proof of \cite[Theorem 4.1]{GKtams}. More precisely, in this way we obtain the estimate  
$$
\sup_{B_{\frac{r}{2}}(x_0)} u \leq \bar{k},
$$
where  
$$
\bar{k} = \delta \, \mathrm{Tail}\Big(u_+; x_0, \frac{r}{2}\Big) 
+ c_0^{\frac{1}{\beta}} b^{\frac{1}{\beta^2}} 
\left( \fint_{B_r(x_0)} u_+^p \, dx \right)^{\frac{1}{p}} 
+ r^{\frac{p-n}{p-1}} R^{\frac{n(q-p)}{q(p-1)}} 
\| f \|_{L^{\frac{q}{p}}(B_R(x_0))}^{\frac{1}{p-1}},
$$
with  
$$
c_0 = c(n,p,q,s,C_1,C_2,\Lambda) \, \delta^{\frac{(1-p)\kappa}{p}}, 
\quad b = 2^{\left(\frac{n+ps+p-1}{p} + \frac{\kappa-1}{\kappa}\right)\kappa}, 
\quad \beta = \kappa - 1, 
\quad \kappa = \frac{n}{n-p}.
$$

\subsection{Proof of Theorem \ref{mthm2}}
Let $0 < \rho < r$. Then, by Theorem \ref{mthmbdd} and Lemma \ref{tailest}, for every $\delta \in (0,1]$, there exists a positive constant $c = c(n,p,q,s,C_1,C_2,\Lambda)$ such that  
\begin{equation}\label{mthm2-1}
\begin{split}
\sup_{B_{\frac{\rho}{2}}(x_0)} u 
&\leq \delta \, \mathrm{Tail}\Big(u_+; x_0, \frac{\rho}{2}\Big) 
+ c \, \delta^{-\gamma} \left( \fint_{B_{\rho}(x_0)} u^p \, dx \right)^{\frac{1}{p}} 
+ \rho^\frac{p-n}{p-1}R^\frac{n(q-p)}{q(p-1)}\|f\|_{L^{\frac{q}{p}}(B_\rho(x_0))}^{\frac{1}{p-1}} \\
&\leq \delta \Big( c \sup_{B_\rho(x_0)} u + c \rho^{p'} R^{-p'} \mathrm{Tail}(u_-; x_0, R)\Big) + c \, \delta^{-\gamma} \left( \fint_{B_\rho(x_0)} u^p \, dx \right)^{\frac{1}{p}} \\
&\qquad + c\rho^\frac{p-n}{p-1}R^\frac{n(q-p)}{q(p-1)} \|f\|_{L^{\frac{q}{p}}(B_R(x_0))}^{\frac{1}{p-1}},
\end{split}
\end{equation}
where $\gamma = \frac{(p-1)\kappa}{p(\kappa-1)}$ with $\kappa = \frac{n}{n-p}$.  

Let $\frac{1}{2} \leq \sigma' < \sigma \leq 1$ and $\rho = (\sigma - \sigma')r$. Using a covering argument, we obtain  
\begin{align*}
\sup_{B_{\sigma' r}(x_0)} u 
&\leq c \, \frac{\delta^{-\gamma}}{(\sigma - \sigma')^{\frac{n}{p}}} 
\left( \fint_{B_{\sigma r}(x_0)} u^p \, dx \right)^{\frac{1}{p}} 
+ c \delta \sup_{B_{\sigma r}(x_0)} u 
+ c \delta r^{p'} R^{-p'} \mathrm{Tail}(u_-; x_0, R) \\
&\qquad+ c\frac{r^\frac{p-n}{p-1}R^\frac{n(q-p)}{q(p-1)}}{(\sigma-\sigma')^\frac{n-p}{p-1}}\|f\|_{L^{\frac{q}{p}}(B_R(x_0))}^{\frac{1}{p-1}} \\
&\leq c \, \frac{\delta^{-\gamma}}{(\sigma - \sigma')^{\frac{n}{p}}} 
\big( \sup_{B_{\sigma r}(x_0)} u \big)^{\frac{p-t}{p}} 
\left( \fint_{B_{\sigma r}(x_0)} u^t \, dx \right)^{\frac{1}{t}} 
+ c \delta \sup_{B_{\sigma r}(x_0)} u 
+ c \delta r^{p'} R^{-p'} \mathrm{Tail}(u_-; x_0, R) \\
&\qquad + c\frac{r^\frac{p-n}{p-1}R^\frac{n(q-p)}{q(p-1)}}{(\sigma-\sigma')^\frac{n-p}{p-1}}\|f\|_{L^{\frac{q}{p}}(B_R(x_0))}^{\frac{1}{p-1}},
\end{align*}
for every $t \in (0,p)$ with a constant $c = c(n,p,q,s,C_1,C_2,\Lambda) > 1$.  

Applying Young's inequality with exponents $\frac{p}{t}$ and $\frac{p}{p-t}$ to the above estimate and choosing $\delta = \frac{1}{4c}\in(0,1]$, we obtain  
\begin{equation}\label{heqn}
\sup_{B_{\sigma' r}(x_0)} u 
\leq \frac{1}{2} \sup_{B_{\sigma r}(x_0)} u 
+ \frac{c}{(\sigma - \sigma')^m} \Bigg\{\Big( \fint_{B_r(x_0)} u^t \, dx \Big)^{\frac{1}{t}} + r^\frac{p-n}{p-1}R^\frac{n(q-p)}{q(p-1)}\|f\|_{L^{\frac{q}{p}}(B_R(x_0))}^{\frac{1}{p-1}}\Bigg\}
+ c r^{p'} R^{-p'} \mathrm{Tail}(u_-; x_0, R),
\end{equation}
where $m=\max\{\frac{n}{t},\frac{p-n}{p-1}\}$ for every $t \in (0,p)$ with a constant $c = c(n,p,q,s,C_1,C_2,\Lambda) > 0$.  

Using Lemma \ref{ite} in the above estimate, it follows that  
$$
\sup_{B_{\frac{r}{2}}(x_0)} u 
\leq c \left( \fint_{B_r(x_0)} u^t \, dx \right)^{\frac{1}{t}} 
+ c r^{p'} R^{-p'} \mathrm{Tail}(u_-; x_0, R) 
+ cr^\frac{p-n}{p-1}R^\frac{n(q-p)}{q(p-1)}\|f\|_{L^{\frac{q}{p}}(B_R(x_0))}^{\frac{1}{p-1}},
$$
for every $t \in (0,p)$ with a constant $c = c(n,p,q,s,C_1,C_2,\Lambda,t) > 0$.  

Let $\epsilon>0$, then choosing  
$$
d = r^{p'} R^{-p'} \mathrm{Tail}(u_-; x_0, R) 
+ r^{\frac{p-n}{p-1}} R^{\frac{n(q-p)}{q(p-1)}} 
\|f\|_{L^{\frac{q}{p}}(B_R(x_0))}^{\frac{1}{p-1}} + \epsilon,
$$
in the estimate \eqref{pwh1} and letting $\epsilon \to 0$, and combining the resulting estimate with the above inequality \eqref{heqn}, the theorem follows.

\section{Acknowledgement}
The author thanks IISER Berhampur for the seed grant: IISERBPR/RD/OO/2024/15, Date: February
08, 2024.

\noindent {\textsf{Prashanta Garain\\Department of Mathematical Sciences\\
Indian Institute of Science Education and Research Berhampur\\ Permanent Campus, At/Po:-Laudigam,\\
Dist.-Ganjam, Odisha, India-760003, 
}\\ 
\textsf{e-mail}: pgarain92@gmail.com\\


\medskip
\begin{thebibliography}{10}

\bibitem{Siva}
Siva Athreya and Koushik Ramachandran.
\newblock Harnack inequality for non-local {S}chr\"odinger operators.
\newblock {\em Potential Anal.}, 48(4):515--551, 2018.

\bibitem{BGK}
Agnid Banerjee, Prashanta Garain, and Juha Kinnunen.
\newblock Some local properties of subsolution and supersolutions for a doubly nonlinear nonlocal {$p$}-{L}aplace equation.
\newblock {\em Ann. Mat. Pura Appl. (4)}, 201(4):1717--1751, 2022.

\bibitem{Valcpde}
Stefano Biagi, Serena Dipierro, Enrico Valdinoci, and Eugenio Vecchi.
\newblock Mixed local and nonlocal elliptic operators: regularity and maximum principles.
\newblock {\em Comm. Partial Differential Equations}, 47(3):585--629, 2022.

\bibitem{Hong}
Stefano Biagi, Serena Dipierro, Enrico Valdinoci, and Eugenio Vecchi.
\newblock A {H}ong-{K}rahn-{S}zeg\"o{} inequality for mixed local and nonlocal operators.
\newblock {\em Math. Eng.}, 5(1):Paper No. 014, 25, 2023.

\bibitem{Biagicvpde}
Stefano Biagi and Eugenio Vecchi.
\newblock Multiplicity of positive solutions for mixed local-nonlocal singular critical problems.
\newblock {\em Calc. Var. Partial Differential Equations}, 63(9):Paper No. 221, 45, 2024.

\bibitem{Biswasnodea}
Nirjan Biswas and Harsh Prasad.
\newblock Lipschitz potential estimates for diffusion with jumps.
\newblock {\em NoDEA Nonlinear Differential Equations Appl.}, 32(5):Paper No. 88, 26, 2025.

\bibitem{BoGi}
E.~Bombieri and E.~Giusti.
\newblock Harnack's inequality for elliptic differential equations on minimal surfaces.
\newblock {\em Invent. Math.}, 15:24--46, 1972.

\bibitem{Silva}
S.~Buccheri, J.~V. da~Silva, and L.~H. de~Miranda.
\newblock A system of local/nonlocal {$p$}-{L}aplacians: the eigenvalue problem and its asymptotic limit as {$p\to\infty$}.
\newblock {\em Asymptot. Anal.}, 128(2):149--181, 2022.

\bibitem{Byuncvpde2}
Sun-Sig Byun, Deepak Kumar, and Ho-Sik Lee.
\newblock Global gradient estimates for the mixed local and nonlocal problems with measurable nonlinearities.
\newblock {\em Calc. Var. Partial Differential Equations}, 63(2):Paper No. 27, 48, 2024.

\bibitem{Byuncvpde}
Sun-Sig Byun and Kyeong Song.
\newblock Mixed local and nonlocal equations with measure data.
\newblock {\em Calc. Var. Partial Differential Equations}, 62(1):Paper No. 14, 35, 2023.

\bibitem{Chakeretal}
Jamil Chaker and Minhyun Kim.
\newblock Regularity estimates for fractional orthotropic {$p$}-{L}aplacians of mixed order.
\newblock {\em Adv. Nonlinear Anal.}, 11(1):1307--1331, 2022.

\bibitem{Chen}
Zhen-Qing Chen, Panki Kim, Renming Song, and Zoran Vondra\v~cek.
\newblock Boundary {H}arnack principle for {$\Delta+\Delta^{\alpha/2}$}.
\newblock {\em Trans. Amer. Math. Soc.}, 364(8):4169--4205, 2012.

\bibitem{Iwonajlms}
Iwona Chlebicka, Kyeong Song, Yeonghun Youn, and Anna Zatorska-Goldstein.
\newblock Riesz potential estimates for mixed local-nonlocal problems with measure data, 
{\newblock {\em J. Lond. Math. Soc.} (2) 112 (2025), no. 4, Paper No. e70310, 47 pp.}

\bibitem{FMma}
Cristiana De~Filippis and Giuseppe Mingione.
\newblock Gradient regularity in mixed local and nonlocal problems.
\newblock {\em Math. Ann.}, 388(1):261--328, 2024.

\bibitem{DKPahp}
Agnese Di~Castro, Tuomo Kuusi, and Giampiero Palatucci.
\newblock Local behavior of fractional {$p$}-minimizers.
\newblock {\em Ann. Inst. H. Poincar\'e{} C Anal. Non Lin\'eaire}, 33(5):1279--1299, 2016.

\bibitem{Hitchhiker'sguide}
Eleonora Di~Nezza, Giampiero Palatucci, and Enrico Valdinoci.
\newblock Hitchhiker's guide to the fractional {S}obolev spaces.
\newblock {\em Bull. Sci. Math.}, 136(5):521--573, 2012.

\bibitem{Valpi}
Serena Dipierro and Enrico Valdinoci.
\newblock Description of an ecological niche for a mixed local/nonlocal dispersal: an evolution equation and a new {N}eumann condition arising from the superposition of {B}rownian and {L}\'evy processes.
\newblock {\em Phys. A}, 575:Paper No. 126052, 20, 2021.

\bibitem{Kmnine}
Bart\l~omiej Dyda and Moritz Kassmann.
\newblock On weighted {P}oincar\'e{} inequalities.
\newblock {\em Ann. Acad. Sci. Fenn. Math.}, 38(2):721--726, 2013.

\bibitem{Evans}
Lawrence~C. Evans.
\newblock {\em Partial differential equations}, volume~19 of {\em Graduate Studies in Mathematics}.
\newblock American Mathematical Society, Providence, RI, 1998.

\bibitem{Kasscpde}
Matthieu Felsinger and Moritz Kassmann.
\newblock Local regularity for parabolic nonlocal operators.
\newblock {\em Comm. Partial Differential Equations}, 38(9):1539--1573, 2013.

\bibitem{Fo}
Mohammud Foondun.
\newblock Heat kernel estimates and {H}arnack inequalities for some {D}irichlet forms with non-local part.
\newblock {\em Electron. J. Probab.}, 14:no. 11, 314--340, 2009.

\bibitem{Garainna}
Prashanta Garain.
\newblock Some qualitative and quantitative properties of weak solutions to mixed anisotropic and nonlocal quasilinear elliptic and doubly nonlinear parabolic equations.
\newblock {\em Nonlinear Anal.}, 256:Paper No. 113796, 31, 2025.

\bibitem{GKK}
Prashanta Garain, Wontae Kim, and Juha Kinnunen.
\newblock On the regularity theory for mixed anisotropic and nonlocal {$p$}-{L}aplace equations and its applications to singular problems.
\newblock {\em Forum Math.}, 36(3):697--715, 2024.

\bibitem{GKtams}
Prashanta Garain and Juha Kinnunen.
\newblock On the regularity theory for mixed local and nonlocal quasilinear elliptic equations.
\newblock {\em Trans. Amer. Math. Soc.}, 375(8):5393--5423, 2022.

\bibitem{GKjde}
Prashanta Garain and Juha Kinnunen.
\newblock Weak {H}arnack inequality for a mixed local and nonlocal parabolic equation.
\newblock {\em J. Differential Equations}, 360:373--406, 2023.

\bibitem{GLcvpde}
Prashanta Garain and Erik Lindgren.
\newblock Higher {H}\"older regularity for mixed local and nonlocal degenerate elliptic equations.
\newblock {\em Calc. Var. Partial Differential Equations}, 62(2):Paper No. 67, 36, 2023.

\bibitem{GGactamath}
Mariano Giaquinta and Enrico Giusti.
\newblock On the regularity of the minima of variational integrals.
\newblock {\em Acta Math.}, 148:31--46, 1982.

\bibitem{JN}
F.~John and L.~Nirenberg.
\newblock On functions of bounded mean oscillation.
\newblock {\em Comm. Pure Appl. Math.}, 14:415--426, 1961.

\bibitem{KK}
Juha Kinnunen and Tuomo Kuusi.
\newblock Local behaviour of solutions to doubly nonlinear parabolic equations.
\newblock {\em Math. Ann.}, 337(3):705--728, 2007.

\bibitem{Maly}
Jan Mal\'y and William~P. Ziemer.
\newblock {\em Fine regularity of solutions of elliptic partial differential equations}, volume~51 of {\em Mathematical Surveys and Monographs}.
\newblock American Mathematical Society, Providence, RI, 1997.

\bibitem{Ahmedrmi}
Ahmed Mohammed.
\newblock Harnack's inequality for solutions of some degenerate elliptic equations.
\newblock {\em Rev. Mat. Iberoamericana}, 18(2):325--354, 2002.

\bibitem{Moser}
J\"urgen Moser.
\newblock On {H}arnack's theorem for elliptic differential equations.
\newblock {\em Comm. Pure Appl. Math.}, 14:577--591, 1961.

\bibitem{Nakamuracvpde}
Kenta Nakamura.
\newblock Harnack's estimate for a mixed local-nonlocal doubly nonlinear parabolic equation.
\newblock {\em Calc. Var. Partial Differential Equations}, 62(2):Paper No. 40, 45, 2023.

\bibitem{Valmz}
Xifeng Su, Enrico Valdinoci, Yuanhong Wei, and Jiwen Zhang.
\newblock Regularity results for solutions of mixed local and nonlocal elliptic equations.
\newblock {\em Math. Z.}, 302(3):1855--1878, 2022.

\bibitem{Valjde}
Xifeng Su, Enrico Valdinoci, Yuanhong Wei, and Jiwen Zhang.
\newblock On some regularity properties of mixed local and nonlocal elliptic equations.
\newblock {\em J. Differential Equations}, 416:576--613, 2025.

\bibitem{Zhangbdd}
Jiaxiang Zhang and Shenzhou Zheng.
\newblock The solvability and regularity results for elliptic equations involving mixed local and nonlocal {$p$}-{L}aplacian.
\newblock {\em J. Elliptic Parabol. Equ.}, 10(2):1097--1122, 2024.

\end{thebibliography}
\end{document}